\documentclass{article}
\usepackage{graphicx}

\usepackage[fleqn]{amsmath}

\usepackage{amsthm}
\usepackage{amssymb}
\usepackage{amsbsy}
\usepackage{amsfonts}
\usepackage{mathrsfs}
\usepackage{amsmath}
\usepackage[all]{xy}
\usepackage{amstext}
\usepackage{amscd}
\usepackage[dvips]{epsfig}
\usepackage{psfrag}
\usepackage{enumerate}
\usepackage{flafter}
\allowdisplaybreaks

\theoremstyle{plain}
\newtheorem{thm}{Theorem}[section]
\newtheorem{prop}[thm]{Proposition}
\newtheorem{cor}[thm]{Corollary}
\newtheorem{lem}[thm]{Lemma}
\theoremstyle{definition}

\def\Ker{\mathop{\mathrm{Ker}}\nolimits}

\def\Hom{\mathop{\mathrm{Hom}}\nolimits}

\newcommand{\SK}{{S^3 \setminus K }}

\newcommand{\lra}{\longrightarrow}
\newcommand{\ra}{\rightarrow}
\newcommand{\Q}{{\Bbb Q}}

\newcommand{\Z}{{\Bbb Z}}

\newcommand{\pc}[2]{\mbox{$\begin{array}{c}
\includegraphics[scale=#2]{#1.eps}
\end{array}$}}
\begin{document}
\large
\begin{center}
{\bf\Large Some comparisons of Blanchfield pairings and cohomology pairings of knots}
\end{center}
\begin{center}{Takefumi Nosaka\footnote{
E-mail address: {\tt nosaka@math.titech.ac.jp}
}}\end{center}
\begin{abstract}\baselineskip=12pt \noindent
We study
%provide a close relation between
%reduction from formula of computing diagrammatically
some comparison between a bilinear cohomology pairing in local coefficients and
the Blanchfield pairing of a knot.
We show that the former pairing is an $S$-equivalent invariant, 
and give a criterion to a relation between the two pairings.
We also observe
that the pairings of some knots are equivalent,
and that the pairings of other knots are not equivalent.
% As a result, in some cases, we give diagrammatic computations of the Blanchfield pairing.

%every links in the 3-sphere.
%As a corollary, we propose a canonical construction of bilinear forms on the twisted Alexander modules of links, which we call twisted Blanchfield pairings.
%We also compute some examples.
\end{abstract}

\begin{center}
\normalsize
\baselineskip=11pt
{\bf Keywords} \\
\ \ \ Cup product, knot, Blanchfield pairing, infinite covering, quandle \ \ \
\end{center}

\large
\baselineskip=16pt

\section{Introduction}
\subsection{Motivation and background}\label{jiji1}
The interaction between cup products and intersection forms is a basic method to essentially analyse a $C^{\infty}$-manifold $Y$ in the history (e.g., the classification theorem of simply connected manifolds).
As a typical example, as seen in the Poincar\'{e} duality with trivial coefficients,
non-degeneracy of the intersection form can be shown from the view of the cup product.
Here, it is worth noting that the interactions are implicitly reflected on algebraic futures of the coefficients $\Z$ and $\Z/p$.

%The intersection form (with some charactristic classes) also plays a key role to classify simply connected compact manifolds up to homeomorphism.
%Roughly speaking, simply connected compact manifolds are (almost) determined by their intersection forms up to homeomorphism.
% Such linear forms and their non-degeneracy have been extensively studied in topology.
However, once the (co)homology groups are investigated with local coefficients $A$,
such an interaction has many unknown aspects with ambiguity and differences.
Actually, graded commutativity and the Krull dimension of $A$ appear as obstructions:
%and we readily face %Actually,
for example, we can perceive such a difference even from some dualities of infinite cyclic covering spaces $\widetilde{Y}$ of closed manifolds;
while Blanchfield duality \cite{Bla} on the homology of $\widetilde{Y}$ is defined from some Bockstein operator to realize hermitian intersection forms over $\Z[t^{\pm 1}]$,
Milnor duality \cite{Mil} is anti-hermitianly constructed with regard to the cup products of $\widetilde{Y}$ over fields and requires some assumptions;
moreover, there are not so many descriptions to explicitly connect the two dualities (cf. Theorem \ref{clAl22}; however, partial connections on signatures can be seen in \cite{Kea2,MP}).
%from the viewpoint \eqref{kiso} with infinite cyclic covers of $\widetilde{Y}$ over fields. More generally,

%however,
%while the Cochran-Orr-Teicher theory \cite{COT} on a solvable filtration of the knot concordance group
%was inspired by the $\mathbb{L}$-theory, is proven by some Bockstein operator to realize intersection forms;
%Actually,

\subsection{Settings: cohomology pairings and the Blanchfield duality}\label{jiji2}
In this paper, we focus on the case $Y$ is a complement $S^3 \setminus K$ of a knot $K$, and
study a relation between cohomology pairings and the Blanchfield duality.
The former pairing
%interpret three classical 3-manifold invariants from twisted cohomology pairings,
%and provide a method of diagrammatically computing the invariants.
%The twisted pairing
is constructed from the abelianization $\mathrm{Ab} :\pi_1(Y) \ra \Z= \langle t \rangle$, as follows.
Choosing a Seifert surface $\Sigma$, we regard it as a relative homology $2$-class in $ H_2 ( Y ,\partial Y ;\Z )$.
Set up $\Z[t^{\pm 1}]$-modules $M$ and $M'$, and a sesquilinear\footnote{A bilinear map $\psi : M \times M' \ra A$ over $\Z$ is said to be {\it sesquilinear},
if %this $\psi$ is $G $-invariant and is over $A$.
%Precisely, the equality %ies
$ \psi(t x ,y )= t \psi(x , y ) = \psi(x ,t^{-1} y )$ holds for any $x\in M,y \in M'$. } %$t$-invariant
bilinear function $ \psi : M \times M' \ra A $ for some $\Z[t^{\pm 1}]$-module $A$.
Then, we can define the pairing as a bilinear form
%as the following composite map $\mathcal {Q}_{\psi}$:
\begin{equation}\label{kiso}
\mathcal {Q}_{\psi}:
H^1( Y,\partial Y ; M ) \otimes H^1( Y,\partial Y ; M ') \xrightarrow{\ \smile \ } H^2( Y ,\partial Y ; M \otimes M' ) \xrightarrow{ \ \bullet \cap \Sigma \ } M \otimes M' \xrightarrow{\ \psi \ }A
. \end{equation}
Here we regard $M$ and $M'$ as the local coefficient modules of $Y$ via $\mathrm{Ab}$, and
the first map $\smile$ is the cup-product, and the second is the cap-product with $ \Sigma $.
Though this pairing $\mathcal {Q}_{\psi}$ seems speculative and uncomputable from definitions,
the author \cite{Nos5} has developed a diagrammatic computation of the $\mathcal {Q}_{\psi}$ (see also \S \ref{Ss3}). %in the link case $Y=S^3 \setminus L$ (see also \S \ref{s3s3w}).
%Since this paper
%Moreover, we emphasize that
%an attempt to reduce link-invariant in terms of cup producets ensures a computation.
%Let us provide four applications of this work to known topological objects,
%

On the other hand, we roughly review the Blanchfield pairing \cite{Bla}
of a knot $K$ with Alexander polynomial $\Delta$.
%Let $f$ be the abelianization $\pi_1( S^3 \setminus K) \ra \Z$, and choose a generator $t$ of this $\Z$.
The first homology $ H_1( S^3 \setminus K ;\Z[t^{\pm 1}] ) $ with local coefficients is called {\it the Alexander module}, i.e.,
the first homology of the covering space $\tilde{Y}$.
Then, from the view of intersection forms in $\tilde{Y}$, the Blanchfield pairing is defined as a sesquilinear form
\begin{equation}\label{blaeq221} \mathrm{Bl}_K: H_1( S^3 \setminus K ;\Z[t^{\pm 1}] )^{\otimes 2} \lra \Z[t^{\pm 1}]/(\Delta) ,
\end{equation}
(\S \ref{99822} reviews the formulation \footnote{In many cases (see \cite{Kaw,T2,Hil}), the image is described as the
injective module $\Q(t)/ \Z[t^{\pm 1}] $. However, such as \cite{T},
%as is known,
the pairing factors through the inclusion $ \Z[t^{\pm 1}]/(\Delta) \hookrightarrow \Q(t)/ \Z[t^{\pm 1}] $ that sends $[f]$ to $[f/ \Delta] $.}).
%$ R$, where $R$ is either the quotient module $ $
%or the quotient ring $ \Z[T^{\pm 1}] /(\Delta_K (T)) $ subject to the 1-st Alexander polynomial.
This $ \mathrm{Bl}_K $ is known to be non-singular, hermitian and sesquilinear (see \cite{Bla,Kaw,Hil});
further, it is a complete invariant of ``the $S$-equivalences" in knots (see \cite{T2,NS} for details). More precisely, two knots $K $ and $K'$ are $S$-equivalent if and only if
the associated pairings $ \mathrm{Bl}_K $ and $ \mathrm{Bl}_{K'} $ are isomorphic as a bilinear form.

\subsection{Main results}\label{jiji3}
Thus, it is natural to ask 
whether $\mathcal{Q}_{\psi } $ is invariant under $S$-equivalence. The main result is as follows:
\begin{thm}\label{cl022}
%Let $M ,A$ be $\Lambda$-modules, and $\psi : M \otimes M^{\rm op} \ra A$ be a bilinear form satisfying %uch that $ \psi(ta,b)= \psi( a, t^{-1} b)= t \psi (a,b)$ for any $a,b$.
If two knots $K $ and $K'$ are $S$-equivalent, then the cohomology pairings $\mathcal{Q}_{\psi } $ and $\mathcal{Q}_{\psi }' $ are equal up to bilinear isomorphisms.
% c as a bilinear.
\end{thm}
\noindent
We put the proof in Section \ref{99822}.
In conclusion, this theorem implies that the cohomology pairing $\mathcal{Q}_{\psi } $
can be described from the $ \mathrm{Bl}_K $
in principle.

The main purpose of this paper is a study of such a description.
For this, % To solve this problem, we need some terminology.
%First, 
we point out that it is reasonable to suppose $M= \Lambda/(\Delta )$, since
there is a $\Z[t^{\pm 1}]$-module isomorphism
\begin{equation}\label{oo00} \kappa : H_1( S^3 \setminus K ;\Z[t^{\pm 1}] ) \cong H^1(S^3 \setminus K, \partial (S^3 \setminus K );M ), 
\end{equation}
which is explicitly described in Section \ref{99822}.
%In this situation, 
Let $\bar{} :  \Z[t^{\pm 1}]  \ra \Z[t^{\pm 1}]  $ be the involution defined by $ \bar{t}=t^{-1}$
The following theorem asserts that a constant multiple of the Blanchfield pairing 
can be recovered from some bilinear form $ \mathcal{Q}_{\psi}$ in some cases (we put the proof in \S \ref{99822}):
\begin{thm}\label{clAl22}
Let $K$ be a knot with Alexander polynomial $\Delta$, and let $M$ be the quotient module $ \Z[t^{\pm 1}]/ ( \Delta) .$
Define $ \psi_0: M \otimes M \ra \Z[t^{\pm 1}]/ ( \Delta) $ by $ \psi_0 (x,y)=\bar{x}y . $

Then, there is a constant $\alpha_K \in \Z[t^{\pm 1}]/ ( \Delta)$ with 
%$ \Upsilon = \alpha_K \Phi$ and 
$ \overline{\alpha_K}=\alpha_K $
%Then, there is a $\Z[t^{\pm 1}]$-module isomorphism
%$$\kappa: H_1( S^3 \setminus K ;\Z[t^{\pm 1}] ) \cong H^1(S^3 \setminus K, \partial (S^3 \setminus K );M )$$
%In particular, if $X$ is of the form $\Z[t^{\pm }]/ ( \Delta_K) $, then
such that 
% as $\Z[T^{\pm 1}]$-modules.
the following equality holds as bilinear forms:
$$ \mathcal{Q}_{\psi_0 } (\kappa(x) ,\kappa(y))= \alpha_K \frac{1+t}{1-t} \cdot \mathrm{Bl}_K(x,y) \in \Z[t^{\pm 1}]/ ( \Delta) , $$
for any $x,y \in H_1(S^3 \setminus K ;\Z[t^{\pm 1}] ) $. Here $\kappa $ is written in \eqref{oo00}.
% the bilinear form $$ is
% equivalent to the $$-multiple of the Blanchfield pairing $$ of the knot $K$.
\end{thm}
In summary, it is fair to state that, this theorem 
gives a cohomological approach to $\mathrm{Bl}_K$ in the sense \ref{blaeq221}, and an obstruction $\alpha_K$ from the approach, in contrast with the previous works \cite{Bla,FP,MP,Kea2} as homological approach. 
However, it is a future problem to ask a relation between the Milnor pairing and our pairing $ \mathcal{Q}_{\psi }$.

We give some remarks on this theorem:
We note that $1-t$ is invertible in $ \Z[t^{\pm 1}]/ ( \Delta)  $ because of $\Delta(1)=\pm 1 $.
%, and that
%whereas we omit constant multiplication thorough $\mathbb{N}$,
%the $(1+t)$-multiple is not so important, because $1+t$ is not zero in the ring $\Z[t^{\pm 1}]/ ( \Delta_K) $ and
%the non-singularity of
%$\mathrm{Bl}_K $ is non-singular.
%???In summary, we now emphasize novel advantages of this theorem.
%First, noting from Remark \ref{ddedrem} that $\mathcal{Q}_{\psi_0 }$ is equivalent to the cup product $\smile_{KL} $ on $\tilde{Y}$
%$\smile_M $ on the infinite cylcic cover of $S^3 \setminus K $,
%considered by Milnor \cite{Mil},
Furthermore, the constant multiple of $(1+t)(1-t)^{-1} $ is a key to connect the hermitian pairing $\mathrm{Bl}_K$ with the anti-one $\mathcal{Q}_{\psi_0 }$.
%Further, this theorem is the reason why the assumption in \cite{Mil}
% enables us to detect the $n_K$ from a concrete computation of $ \mathcal{Q}_{\psi} $.
%Furthermore,
This theorem implies that, on the assumption,
if $\alpha_K$ and $\Delta(-1) \in \Z$ are not zero-divisors in $\Lambda / ( \Delta) $,
then the Blanchfield pairing $\mathrm{Bl}_K $ can be completely recovered from the pairing $\mathcal{Q}_{\psi_0 } $
(see \S \ref{S52} for such examples);
however, conversely, if either of $\alpha_K$ and $\Delta(-1) \in \Z$ is zero-divisor,
all information of $ \mathrm{Bl}_K $ can not be recovered from the pairing $\mathcal{Q}_{\psi_0 } $; % (see \S\S \ref{S51} and \ref{} for such examples).
%Under an assumption, we will show a criterion to completely recover $\mathrm{Bl}_K $ from $\mathcal{Q}_{\psi } $; see Theorem \ref{clAl22} for the detail. For example, in Section \ref{S52}, we will see that every Pretzel knot satisfies the assumption, and give such a complete recovery. However, in general,
Section \ref{S51} observes
some cases where the cup products $\mathcal{Q}_{\psi } $ lose many information of $\mathrm{Bl}_K$, according to complexity of the Alexander module $ H_1( S^3 \setminus K ;\Z[t^{\pm 1}] ) $. % becomes complicate.
%that, when the Alexander module $ H_1( S^3 \setminus K ;\Z[t^{\pm 1}] ) $ is complicate, the cup product $\mathcal{Q}_{\psi } $ lose many information of $\mathrm{Bl}_K $.
In addition, in Section \ref{S53}, we will see that, for even the torus knot,
the problem of the recovery is not so easy.

%(I) The first one is a reduction as for the classical Blanchfield parings of knots. To be more precise (Theorem \ref{clAl22}), if $G$ is the infinite cyclic group $\Z$ and $Y$ is a knot complement, the bilinear form \eqref{kiso} associated with appropriate choices of $\psi$ and $M$ is equal to the $(t+1)/(t-1)$-multiple of the classical pairing;
% This result implies that our twisted Blanchfield paring includes the classical one as a generalization. consequently, the theorem enables us to compute the Blanchfield pairing using no Seifert surface (cf. the formulas \cite{T2} from Seifert surfaces); As a result, we succeed in first determining the pairing of the torus knot; see Proposition \ref{exact}.

Finally, in Appendix \ref{999}, we will see that
%as a corollary,
%we show that
%recalling that the bilinear form is a generalization of the quandle 2-cocycle invariants,
``the cocycle knot-invariants \cite{CJKLS} from Alexander quandles" also turn to be topologically recovered from the Blanchfield pairing (see Theorem \ref{clA21l31} for the details);
%here, we should emphasize that the 2-cocycle invariant \cite{CJKLS} was defined in a combinatorial way, so for (see also Remark \ref{cl31l} for the details).

%invariant$\mathrm{Sign}_w(K,\chi) - \mathrm{Sign}_1(K,\chi)$
%It has been a longstanding problem to provide such a computation of the signature.
%In this paper, as a result from (III) and \cite{Neu}, we succeed in giving a diagrammatic algorithm to practically compute the difference $\mathrm{Sign}_w(K,\chi) - \mathrm{Sign}_1(K,\chi)$; see \S \ref{AAp1} for the detail.
%Since known computations were done in rare conditions (see \cite{Gil} for knots of genus one), the result is expected to be use of the study of the concordance group or slice genus.

This paper is organized as follows.
%Section 2 %introduces bilinear forms on twisted Alexander dual modules, and Section 5
%gives a relation between the Blanchfield pairing and cohomology pairing.
%of a knot and the twisted cup products of infinite cyclic coverings.
Section 2 %studies ??% the the Casson-Gordon local signature.
%Kirk-Livingston 2-form.
% Section 6 proves the main theorems after a review of relative group cohomologies, and Section 4
reviews the computation in \cite{Nos5}, and Section 3 gives the proof of Theorem \ref{clAl22}. %some examples of the computation.
Sections 4--6 give some computations of the pairings.
% formulates the linear forms by means of quandle cocycle invariants,
%and states the main theorems.
%Section 3 describes some computation and states the main theorems.d
%Section 5 completes the proofs of the theorems.
%Appendix A observes an interpretation of \eqref{kiso} from the work of Goldman \cite{G}.
%and Appendix B formulates and discusses a modified Casson-Gordon signature.

\

\noindent
{\bf Conventional notation.} \
Every knot $K$ is understood to be smooth, oriented, and embedded in the 3-sphere $S^3$ as a circle.
We regard the complement $S^3 \setminus K $ as the 3-manifold which is obtained from $S^3$ by removing an open tubular
neighborhood of $K$.
%, i.e., $Y=S^3 \setminus \nu L.$
%By $\# L$ we mean the number of the link component, i.e., $\# L = \pi_0( \partial Y_L).$
%Furthermore, we denote by $ \pi_L $ the fundamental group $\pi_1(S^3 \setminus \nu L)$ of the link complement for short.
% and with identity $1$.
In ordinary papers on the Blanchfield pairing,
by $\Lambda $ we mean the Laurent polynomial ring $\Z[t^{\pm 1}]$,
with involution $\bar{t}=t^{-1}$.
Moreover, we denote a $\Lambda $ module by $M$;
furthermore, 
let $M^{\rm op}$ be $M$ with the $\Lambda $ module structure by $\lambda \cdot m:= \bar{\lambda}m$, where $\lambda \in \Lambda , m \in M^{\rm op}=M$.

\section{Review; diagrammatic computation of the cohomology pairing.}\label{Ss3}
%This section, by composing the theorems in the previous sections with the result in \cite{Nos5}.
This section strictly describes diagrammatic computation of the cohomology pairing,
and gives the proof of Theorem \ref{clAl22}.
%and Sections ref{Lb12r3244}--\ref{AAp1} give some examples.
We will need some knowledge of quandles before proceeding.

Throughout this section, we fix two $\Z[t^{\pm 1} ]$-modules $X$ and $A$.
%Let $Z$ be a subset of $G$ closed under the conjugation operation.
Further, define a binary operation on $X$ by
\begin{equation}\label{kihon} \lhd: X \times X \lra X ; \ \ \ \ \ \ (a,b) \longmapsto t (a-b)+b . \end{equation}
The pair $(X,\lhd)$ is called {\it an Alexander quandle} \cite{CJKLS,Joy,Nos1}.
%which was first introduced in \cite[Lemma 2.2]{IIJO}.
%We remark such quandles.

We review colorings.
Pick a knot $K \subset S^3$ with orientation and an oriented knot diagram $D$ of $K.$
% that the set $ \mathrm{Col}_X(D) $ is bijective to the set of group homomorphisms
A map $\mathcal{C}: \{ \mbox{arcs of $D$} \} \to X$ over $f$ is an $X$-{\it coloring} if
it satisfies %let us define an $X$-{\it coloring over $f$} to be such that
$\mathcal{C}(\alpha_{\tau}) \lhd \mathcal{C}(\beta_{\tau}) = \mathcal{C}(\gamma_{\tau})$ at each crossing of $D$ illustrated as
%$ \mathcal{C}(\gamma_k)= \mathcal{C}(\gamma_i) \lhd \mathcal{C}(\gamma_j)$
Figure \ref{koutenpn}. %\ref{koutenpn}.
Let $\mathrm{Col}_X(D) $ denote the set of all $X$-colorings.
%It is worth noticing that
By definition, this $\mathrm{Col}_X(D ) $ canonically injects into the product $X^{\alpha (D)}$, where $\mathrm{Arc} (D)$ is the number of arcs of $D$.
Therefore, $\mathrm{Col}_X(D)$ %is a subset of $M^m$,
%which
serves as a $\Lambda$-submodule of $X^{ \mathrm{Arc} (D)}$.
Furthermore, the diagonal submodule $X_{\rm diag} \subset M^{ \alpha (D)}$
is contained in $\mathrm{Col}_X(D)$, and is a direct summand of $\mathrm{Col}_X(D )$.
%we call a constant map $\mathcal{C}: \{ \mbox{arcs of $D$} \} \to X$ over $f$ {\it trivial}.
%Since the set of such trivial maps turns out to be
Thus, we denote by $\mathrm{Col}^{\rm red}_X(D)$ another direct summand, i.e.,
$\mathrm{Col}_X(D ) \cong X_{\rm diag} \oplus \mathrm{Col}^{\rm red}_X(D). $

% according to
%In particular, by the linear operation \eqref{kihon}.

\vskip 0.419937pc

\begin{figure}[htpb]
\begin{center}
\begin{picture}(100,26)
%\put(-41,2){\Large $R_{\tau} $}
%\put(86,2){\Large $R_{\tau} $}
\put(-22,27){\large $\alpha_{\tau} $}
\put(13,25){\large $\beta_{\tau} $}
\put(13,-13){\large $\gamma_{\tau} $}
%\put(13,-8){\large $\gamma_k $}

\put(-36,3){\pc{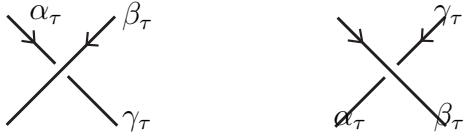}{0.27304}}

\put(131,-13){\large $\beta_{\tau} $}
%\put(139,27){\large $\gamma_k $}
\put(93,-13){\large $\alpha_{\tau} $}
%\put(139,27){\large $\gamma_k $}
\put(131,27){\large $\gamma_{\tau} $}
\end{picture}
\end{center}
\caption{\label{koutenpn} Positive and negative crossings, where arcs are assigned by $A$.}
\end{figure}

\vskip 0.995pc

% $ \longleftrightarrow \mathrm{Hom}_{\rm gr}(\pi_K , G ). \end{equation}
%Furthermore, the image of .
%As is well-known ,
%The set ${\rm Col}_X(D)$ is known to be bijective to $\Hom_{\rm Qnd}(Q_K,X) $; see \cite[\S ??]{Joy}.
%the set of quandle homomorphisms Q_K \ra X$.

Furthermore, one introduces a sesquilinear form on the $\Lambda $-module $ \mathrm{Col}_X(D)$ as follows. 
% (Definition \ref{deals3}).
Take another Alexander quandle $X'.$
Let $\psi : X \times X' \ra A$ be a sesquilinear map over $\Z$.
Define a map
\begin{equation}\label{aaa} \mathcal{Q}_{D, \psi } : \mathrm{Col}_X(D) \times \mathrm{Col}_{X'}(D) \lra A; \ \ \ ( \mathcal{C}, \mathcal{C}' ) \longmapsto \sum_{\tau} \epsilon_{\tau} \cdot \psi \bigl( \mathcal{C}(\alpha_{\tau}) - \mathcal{C}(\beta_{\tau} ) ,
\ \mathcal{C}'(\beta_{\tau}) (1-t^{-1}) \bigr),\end{equation}
where $\tau $ runs over all the crossings of $D$, and
the symbols $\alpha_{\tau} , \ \beta_{\tau} $ are the arcs and $\epsilon_{\tau} \in \{ \pm 1\}$ is the sign of $\tau$ according to Figure \ref{koutenpn}.
The sum in \eqref{aaa} is sometimes called {\it a weight sum}.
%This $\mu_{\mathcal{C}}$ is called {\it the fundamental class of $\mathcal{C}$}, and is known to be a 2-cycle in ``the quandle complex $C_2^Q(X;\Z)$".
%Given a quandle 2-cocycle $\psi$,
%the {\it quandle cocycle invariant} is defined to be a map $\mathcal{Q}_{\psi,D}$ which takes
%an $X$-coloring $\mathcal{C}$ to the coupling between the sum $\mu_{\mathcal{C}}$ and $\psi$; Namely
%\begin{equation}\label{cyc52} \mathcal{Q}_{\psi,D}: \mathrm{Col}_X(D) \lra A ; \ \ \ \ \mathcal{C} \longmapsto \langle \psi , \mu_{\mathcal{C}} \rangle. \end{equation}
%in a &%is can be computedrelatively easy way;
%by definitions;
Then, the sesquilinear form $ \mathcal{Q}_{D, \psi } $ is topologically detected by the following sense:
\begin{thm}[{A special case of \cite[Theorem 2.2]{Nos5}}]\label{mainthm1}
Let $E_K= \SK$. Let $M$ be $X$ and $M'$ be $X'$ as above.
Then, there are $\Lambda$-module isomorphisms
%$\Upsilon$ is a splitting surjection whose kernel is $M$. In particular, we have
$$ \mathrm{Col}_{X } (D) \cong H^1(E_K , \ \partial E_K ;M ) \oplus M, \ \ \ \ \ \mathrm{Col}_{X }^{\rm red} (D) \cong H^1(E_K , \ \partial E_K ;M ) . $$
Furthermore, on the isomorphisms, the restriction of $ \mathcal{Q}_{\psi,D}$ on $ \mathrm{Col}_{X }^{\rm red} (D) \times
\mathrm{Col}_{X' }^{\rm red} (D)$ is
equal to the bilinear cohomology pairing $\mathcal{Q}_{ \psi }$ in \eqref{kiso}.
%following composite: $$ H^1( E_K ,\partial E_K ; M )^{\otimes 2} \xrightarrow{\ \ \smile \ \ } H^2( E_K ,\partial E_K ; M^{\rm op}\otimes M ) \xrightarrow{ \ \langle \bullet , \mu \rangle \ } M^{\rm op} \otimes M \xrightarrow{ \ \ \langle \psi, \bullet \rangle \ \ }A . $$
\end{thm}

To summarize, the point is that, given a diagram $D$,
we can diagrammatically compute the form $ \mathcal{Q}_{\psi,D }$ by definitions,
and %. Here, the point is
that we need no description of the Seifert surface $\Sigma$; in a comparison, 
there are approaches to the signature of knots without using Seifert surfaces \cite{Kea2,MP}.

\subsection{Proof of Theorem \ref{cl022}.}\label{Ss375}
To prove the invariance
of the cohomology pairing \eqref{kiso} under $S$-equivalence, 
% the $S$-equivalences of the cohomology pairing \eqref{kiso},
%Furthermore,
we review the $S$-equivalence of knots.
%As is known,
While there are several definitions of the $S$-equivalences (see \cite{T2,Lic}), this paper uses the definition in the sense of \cite{NS}.
Two knots $K$ and $ K'$ are {\it $S$-equivalent} if
they are related by a finite sequence of the (double delta) local moves shown in Figure \ref{fig.color22}.
Furthermore, Trotter \cite{T2} showed that $K$ and $ K'$ are $S$-equivalent if and only if the associated Blanchfield pairings are isomorphic as bilinear forms.
%Following the notation,

\begin{figure}[htpb]
\begin{center}
\begin{picture}(100,70)
\put(-72,23){\pc{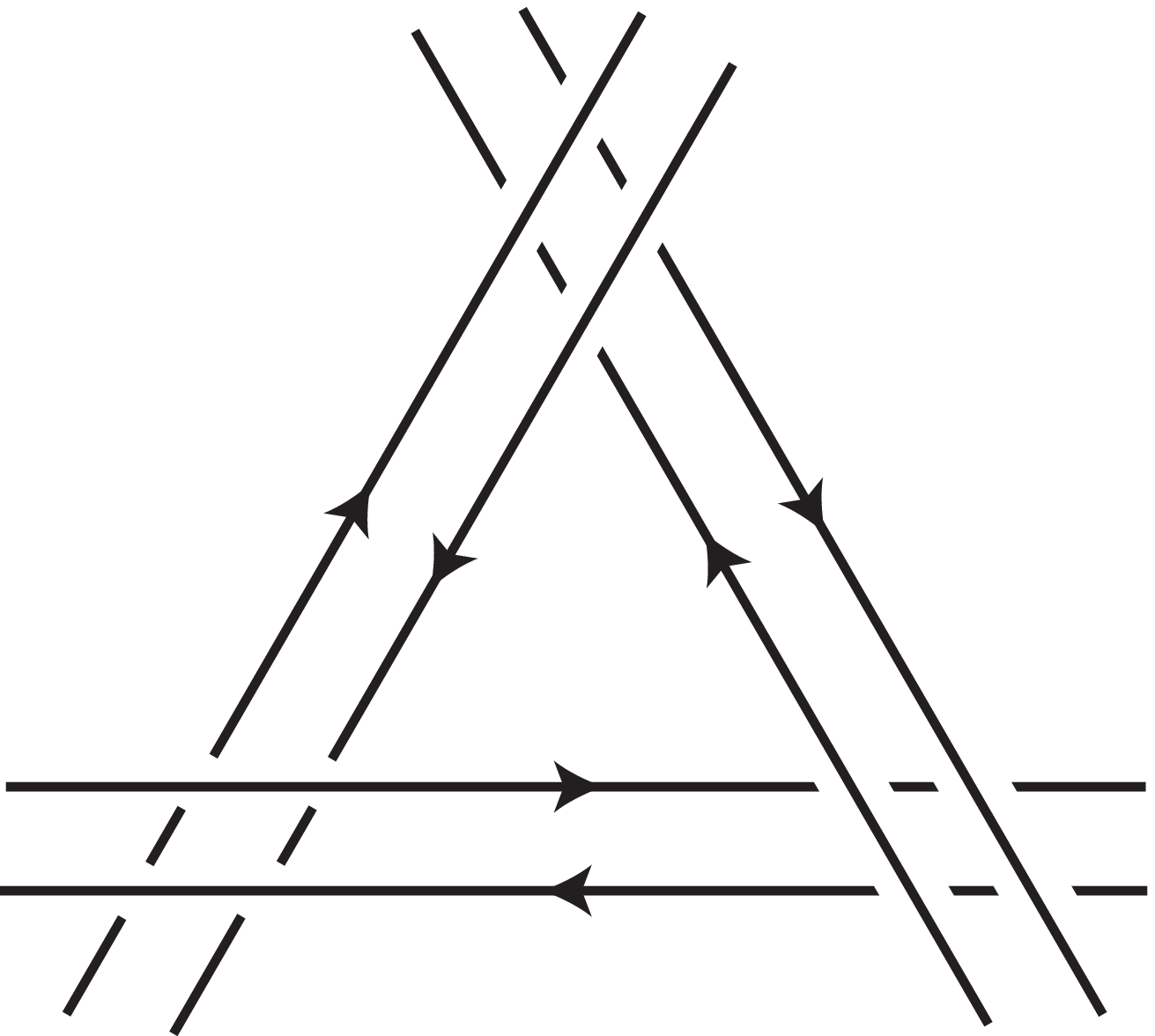}{0.1994}}
\put(92,23){\pc{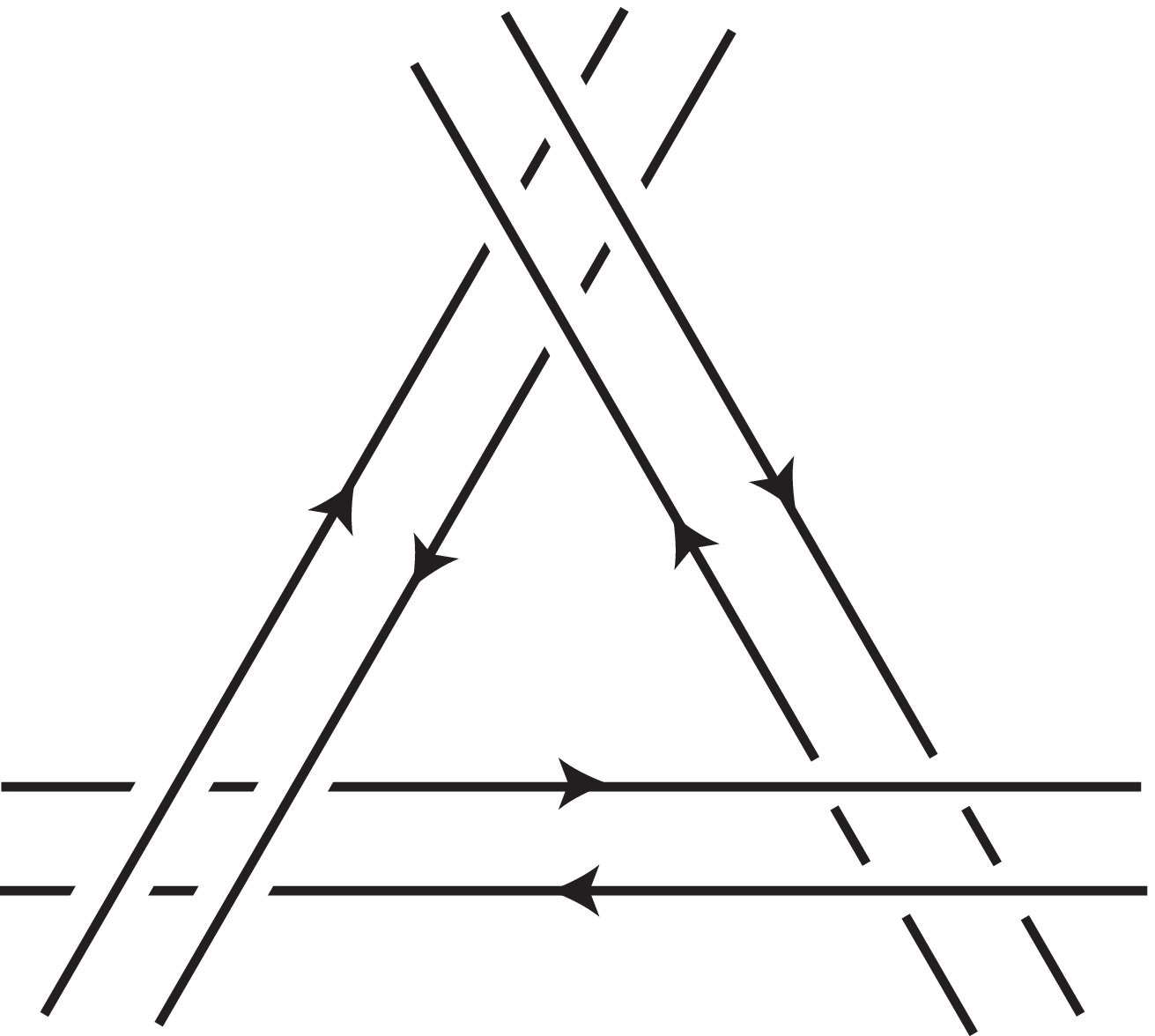}{0.19924}}

\put(-78,1){\small $\alpha_1$}
\put(-78,10){\small $\alpha_2$}
\put(-49,54){\small $\alpha_3$}
\put(-40,63){\small $\alpha_4$}
\put(-26,63){\small $\alpha_5$}
\put(-18,54){\small $\alpha_6$}
\put(10,1){\small $\alpha_8$}
\put(10,10){\small $\alpha_7$}
\put(-57,-11.5){\small $\alpha_{11}$}
\put(-74,-11.5){\small $\alpha_{12}$}
\put(-10,-11.5){\small $\alpha_{10}$}
\put(5,-11.5){\small $\alpha_9$}

\put(-78,50){\LARGE $D$}
\put(178,50){\LARGE $D' $}

\put(82,1){\small $\beta_1$}
\put(82,10){\small $\beta_2$}
\put(114,54){\small $\beta_3$}
\put(124,63){\small $\beta_4$}
\put(136,63){\small $\beta_5$}
\put(146,54){\small $\beta_6$}
\put(174,1){\small $\beta_8$}
\put(174,10){\small $\beta_7$}
\put(106,-13){\small $\beta_{11}$}
\put(93,-13){\small $\beta_{12}$}
\put(154,-13){\small $\beta_{10}$}
\put(169,-13){\small $\beta_9$}

\put(33,26){\huge $ \longleftrightarrow $}
\end{picture}

\end{center}
\vskip 0.6937pc
\caption{A double delta move with 24 arcs \label{fig.color22}}
%\vskip -0.9937pc
\end{figure}

Since the following lemma is elementary, we omit writing details of the proofs.

\begin{lem}\label{clAl293} Consider an $X$-coloring and an $X'$-coloring of the eight arcs illustrated in the figure below, where the alphabets are elements in $X$ or $X'$.
Then, the weight sum with respect to the four crossings is $\psi \bigl( (1-t) (a-b), \ c'-d' \bigr)\in A.$
\end{lem}\begin{figure}[htpb]
\begin{center}
\begin{picture}(300,80)
\put(33,46){\pc{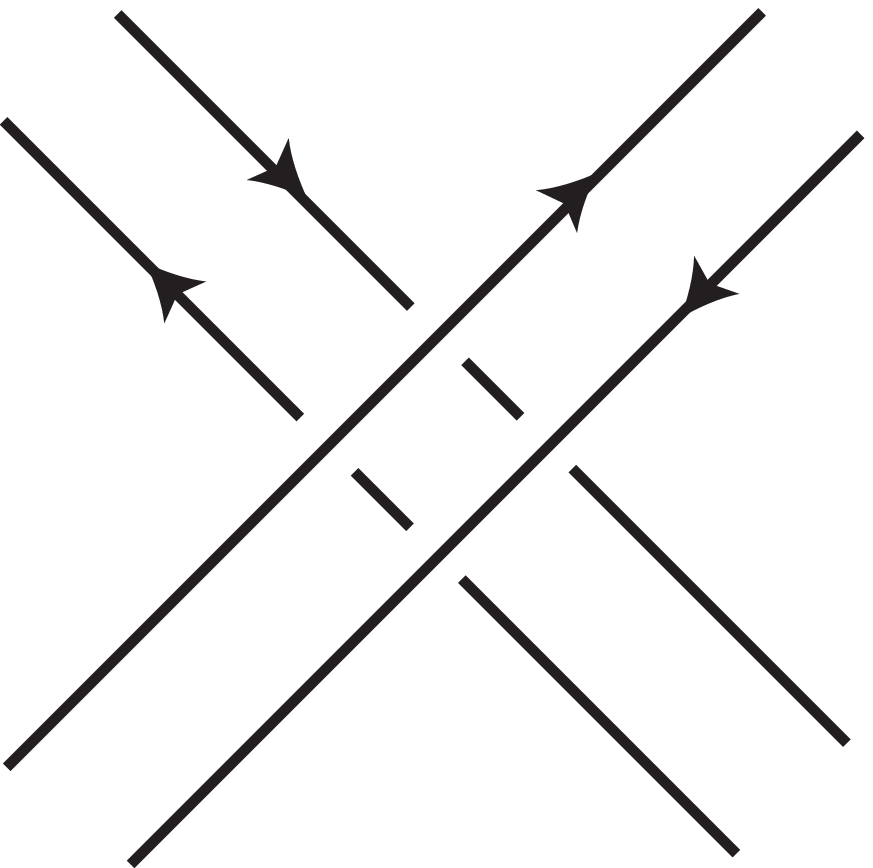}{0.200994}}

\put(39,20){\normalsize $d$}
\put(32,30){\normalsize $c$}
\put(32,65){\normalsize $b$}
\put(39,74){\normalsize $a$}
\put(77,20){\normalsize $u$}
\put(88,30){\normalsize $v$}

\put(233,46){\pc{doubledelta4}{0.200994}}

\put(239,20){\normalsize $d'$}
\put(232,30){\normalsize $c'$}
\put(232,65){\normalsize $b'$}
\put(239,74){\normalsize $a'$}
\put(277,20){\normalsize $u'$}
\put(288,30){\normalsize $v'$}

\end{picture}
\end{center}
\vskip -2.29937pc
% \caption{A double delta move \label{fig.32}}
% \vskip -0.9937pc
\end{figure}

\begin{proof}[Proof of Theorem \ref{cl022}.]
Suppose that two knots $K $ and $K'$ are $S$-equivalent.
Let $D$ and $D'$ be diagrams of $K$ and $K'$, respectively.
% with
By the results mentioned above, we may assume that the difference between $D$ and $D'$ is only a double delta move.
Given an $X$-coloring $\mathcal{C}$ of $D$,
consider the assignment $\mathcal{C}_1$ of $D'$ such that $\mathcal{C}_1(\beta_i) =
\mathcal{C}(\alpha_i) $. We can easily check that $\mathcal{C}_1 $ is an $X$-coloring of $D' $,
and the correspondence $\mathcal{C} \mapsto \mathcal{C}_1 $ gives rise to
a $\Lambda $-isomorphism $\lambda : \mathrm{Col}_X(D) \ra \mathrm{Col}_X(D') $.
Here, we can easily check that 
this $ \lambda$ preserves the direct sum decomposition
$Col_X(D^{\bullet }) = H^1(E_{K^{\bullet }}; \partial E_{K^{\bullet }};M) \oplus M.$
Furthermore, we can easily verify from Lemma \ref{clAl293} the equality $\mathcal{Q}_{\psi,D}(\mathcal{C},\mathcal{C}')= \mathcal{Q}_{\psi,D'}(\lambda (\mathcal{C}),\lambda (\mathcal{C}')) $ for any colorings $\mathcal{C}, \mathcal{C}'. $ In the sequel, the associated bilinear forms $ \mathcal{Q}_{\psi}$ and $\mathcal{Q}_{\psi}'$
are equivalent; hence, so are the corresponding cohomology pairings by Theorem \ref{mainthm1}, as required.
\end{proof}

%\section{Proofs of Theorem \ref{}.}\label{Sproof1}
%We will give the proofs of all the statements in Sections \ref{LFsu32op}-\ref{SCGSS}.

%Changing the subject,
%, by assuming the main theorems \ref{mainthm1} and \ref{mainthm2}.

%The proof appear in \S \ref{}, after a review of a relative group cohomology of $Y_L$.
%Then this bilinear function
%is exttended to
%$$ H^1 (\widetilde{Y}_L , \partial \widetilde{Y}_L ; (\mathbb{F}[t^{\pm}] /\Delta_f )) \times H^1
% (\widetilde{Y}_L , \partial \widetilde{Y}_L ; (\mathbb{F}[t^{\pm}] /\Delta_f )) \lra
% \mathbb{F}[t^{\pm}] /\Delta_f . $$

\section{Proof of Theorem \ref{clAl22}}\label{99822}

The end of this section is devoted to proving Theorem \ref{clAl22}, which gives a trial to recover the Blanchfield pairing from $\mathcal{Q}_{\psi}$.
The reader, who has mainly an interest in examples of computation,
may read only Proposition \ref{bbbn2}, and skip this section.  
In this section,  given a matrix $V$, we denote the transposed matrix by $V'.$
%Let $\Lambda$ denote the ring $\Z [t^{\pm 1}]$, and $E_K$ denote the knot complement $\SK$ in what follows.

\subsection{The Blanchfield pairing from cup product}\label{kkl2}
We first recall the calculation of the Blanchfield pairing in terms of homology \cite{FP,Kea}.
Choose a Seifert surface $F$ of $K$ whose genus is $g$, where
we may assume the existence of a bouquet of circles $W \subset F$ such that $W$ is a deformation retract of $F$
and the inclusion $F \subset S^3$ is isotopic to the standard embedding $W \subset F$.
Then, we have the Seifert form $\alpha : H_1(F;\Z) \otimes H_1(F;\Z) \ra \Z$; see \cite[Chapter 6]{Lic} for the definition.
The matrix presentation is commonly written by $V \in \mathrm{Mat}(2g\times 2g;\Z)$, and is called {\it the Seifert matrix}. %By Poincar\'{e} duality, $V-^t \!\! V \in \mathrm{Mat}(2g\times 2g;\Z) $ is known to be invertible.
\begin{thm}[{\cite{Kea}. See also \cite{FP}\footnote{Strictly speaking, 
the notation of $t$ is that in \cite{Kea,FP}. However, if we replace $t$ by $t^{-1}$, the notations are equal. }}]\label{bbbn24}
The first homology $H_1(E_K ; \Lambda) $ is isomorphic to
the quotient $\Lambda^{2g}/ (tV- V') \Lambda^{2g} $.
The Blanchfield pairing is isomorphic to the bilinear form
$$ \bigl( \Lambda^{2g}/ (tV- V') \Lambda^{2g} \bigr)^2 \lra \Lambda/\Delta; \ \ \ \ \ \ (v,w) \longmapsto
\ (1-t) (\overline{v} (tV- V')^{-1}) w' . $$
\end{thm}
Next, in cohomological terminology, we will reformulate the Blanchfield pairing:
Considering the exact sequence $0 \ra \Lambda \stackrel{\Delta}{\lra} \Lambda \lra \Lambda /(\Delta )\ra 0,$
%\begin{equation}\label{oo0} 0 \ra \Lambda \stackrel{\Delta}{\lra} \Lambda \lra \Lambda /(\Delta )\ra 0.
%\end{equation}
we have the Bockstein map
\begin{equation}\label{oo0}
\beta: H^i(E_K ,\partial E_K ;\Lambda /(\Delta )) \lra H^{i+1}(E_K ,\partial E_K ;\Lambda ).
\end{equation}
Here, it is worth noticing that thie map with $i=1$ is an isomorphism, since
$H^{2}(E_K ,\partial E_K ;\Lambda ) \cong H_{1}(E_K ;\Lambda ) $ is annihilated by the Alexander polynomial $\Delta $; see \cite[Theorem 6.17]{Lic}.
We define {\it the cohomological Blanchfield pairing} to be the bilinear map
$$\mathrm{cBl} :H^1(E_K ,\partial E_K ; \Lambda /(\Delta ))^2 \lra \Lambda /(\Delta ) $$
by setting $\mathrm{cBl} (u,v)= \langle \beta (u) \smile v , [ E_K ,\partial E_K ]\rangle .$
Consider the following kernel:
$$ \Ker (tV- V')_{\Lambda/\Delta}:=  \{w \in (\Lambda/(\Delta ))^{2g} \ | \ (tV- V')w=0 \in (\Lambda/(\Delta))^{2g} \ \} .$$
Furthermore, we introduce two maps
\begin{equation}\label{oo38} \psi_l : \Delta \Lambda^{\rm op} \otimes \Lambda\lra \Z[t^{\pm 1}]/ ( \Delta) ; \ \ \ \ \Delta x\otimes y \longmapsto \bar{x}y ,
\end{equation}
\[\psi_r : \Lambda^{\rm op} \otimes \Delta \Lambda \lra \Z[t^{\pm 1}]/ ( \Delta); \ \ \
z\otimes \Delta w\longmapsto \bar{z}w . \]
\begin{prop}\label{bbbn2}
Choose a section $\mathfrak{s}: (\Lambda/\Delta )^{2g} \ra \Lambda^{2g}$.
The cohomology $H^1(E_K ,\partial E_K ; \Lambda /(\Delta )) $ is isomorphic to $ \Ker (tV- V')_{\Lambda/\Delta} $.
Furthermore, the cohomological Blanchfield pairing is isomorphic to the bilinear form
$$\Ker (A)_{\Lambda/\Delta} \times \Ker (A)_{\Lambda/\Delta}\lra \Lambda/(\Delta); \ \ \ \ \ \ (v,w) \longmapsto
\ (1-t) \psi_l ( (tV- V') \mathfrak{s}(v), w ) . $$
Here, $\psi_l: (\Delta \Lambda^{\rm op} \otimes \Lambda)^{2g} \ra\Lambda/(\Delta) $ is the
direct sum of \eqref{oo38}.
%Here, $\mathrm{adj}(A) $ is the adjugate matrix of $A$, and $\Delta $ is $| H_1 (B_K^d ;\Z)|\in \mathbb{N}. $
\end{prop}
To prove Proposition \ref{bbbn2}, we review from \cite{T} the relative cellular chain complex of $( E_K, \partial E_K)$ with local coefficients $R$. Here, we let $R$ be either $\Lambda$ or $ \Lambda/(\Delta)$.
According to \cite[Proposition 4.1]{T}, the complex is isomorphic to %described as
\begin{equation}\label{oop2}C_*: 0 \ra R \stackrel{\partial_3}{\lra} R^{2g} \oplus R^{2} \stackrel{\partial_2}{\lra} R^{2g}\oplus R \stackrel{\partial_1}{\lra} R \ra 0 .
\end{equation}
Here the differential maps $\partial_*$ have matrix presentations
\[ \partial_3= ( 0, 0, \cdots, 0, 1-t), \ \ \ \ \partial_2= \left( \begin{array}{ccc} tV-V' & 0 & 0 \\ 0 & 1-t & 0 \\\end{array} \right) , \ \ \ \partial_1= (0, \cdots, 0, 1-t)' . \]
%\[ \partial_2= \left( \begin{array}{ccc} tV-V' & 0 & 0 \\ 0 & 1-t & 0 \\\end{array} \right) . \]
Furthermore, we consider the cochain complex $C^*:= \Hom( C_* ;R) $.
Pick a 2-cochain and a 1-cochain of the forms
$$c^2=(f_1,f_2) \in \Hom(R^{2g},R )\oplus\Hom(R^{2},R ) , \ \ \ c^1=(g_1,g_2) \in \Hom(R^{2g},R )\oplus\Hom(R,R ) $$
Let $ c^1 \smile c^2 $ be $(1-t) f_1 \cdot g_1' \in \Hom(R,R )= C^3$.
Then, it is shown \cite{T} that the map $ H^1 \otimes H^2 \ra H^3 $ induced by $\smile : C^1 \otimes C^2 \ra C^3$ coincides with the natural cup product.
\begin{proof}[Proof of Proposition \ref{bbbn2}]
Notice that $ \mathrm{det}(tV- V' )= \Delta$.
By the presentation \eqref{oop2}, we have a canonical isomorphism $H^1(E_K ,\partial E_K ;\Lambda /(\Delta )) \cong
\Ker (tV- V')_{\Lambda/\Delta} $, and can identify the Bockstein map $\beta: H^1(E_K ,\partial E_K ;\Lambda /(\Delta )) \ra H^{2}(E_K ,\partial E_K ;\Lambda )$ with the
mapping $v  \mapsto ( tV- V' ) \mathfrak{s}(v)$. Therefore, by the above formula of the cup product, we readily see $\mathrm{cBl}(u,v) = (1-t) \psi_l ( (tV- V') \mathfrak{s}(u), v ) $ as required.
\end{proof}

Next,
%Before going to the proof,
we will see a corollary.
In general, it is easier to quantitatively compute kernels rather than cokernels. %Let us observe Corollary \ref{bbbn2} below.
Using adjugate matrices, consider the linear map
$$ \kappa: \Lambda^{2g}/(tV- V')\Lambda^{2g} \lra \Ker ( tV- V' )_{\Lambda/\Delta}; \ \ v \longmapsto \mathrm{adj}(tV- V') v .$$
This map is an isomorphism. % if $|\Delta | < \infty $:
Indeed, if we choose a section $\mathfrak{s}: \Lambda^{2g}/(tV- V')\Lambda^{2g} \ra \Lambda^{2g}$,
the inverse mapping is defined by the map $w \mapsto (tV- V') \mathfrak{s}(w)/\Delta$.
In summary, from Theorem \ref{bbbn24} and Proposition \ref{bbbn2}, we immediately have
\begin{cor}\label{bb1787} The map $\kappa$ gives the isomorphism
$$ \kappa: H_1(E_K , \partial E_K ; \Lambda) =\Lambda^{2g}/(tV- V') \cong \Ker (tV- V')_{\Lambda/\Delta}=
H^1(E_K , \partial E_K ; \Lambda/\Delta)$$
such that, for any $ x,y \in H_1(E_K , \partial E_K ; \Lambda ) $,
$$ \mathrm{Bl}_K ( x,y )= \mathrm{cBl}_K ( \kappa(x), \kappa(y) ) \in \Lambda/ (\Delta) . $$
\end{cor}

\subsection{Three Bockstein maps}\label{kkl4}
We further need three Bockstein maps and their properties.
We focus on the case $ M=A=\Lambda/(\Delta ) $.
%consider exact sequences
%, i.e., the adjoint map
% $$ H_1( S^3 \setminus K ;\Z[t^{\pm 1}] ) \ra \Hom ( H_1(S^3 \setminus K ;\Z[t^{\pm 1}] ) ,\ \Z[t^{\pm 1}]/(\Delta_K) )$$ is isomorphic. Furthermore, 
Consider exact sequences
\begin{equation}\label{oo1} 0 \lra \Lambda \stackrel{\Delta}{\lra} \Lambda \lra \Lambda /(\Delta )\lra 0,
\end{equation}
\begin{equation}\label{oo2} 0 \lra \Lambda^{\rm op} \stackrel{\Delta}{\lra} \Lambda^{\rm op} \lra \Lambda^{\rm op} /(\overline{\Delta} )\lra 0.
\end{equation}
The tensor products over $\Z$ canonically give rise to an exact sequence
$$0 \ra (\overline{\Delta}  \Lambda^{\rm op} \otimes \Lambda ) \oplus (\Lambda^{\rm op} \otimes \Delta \Lambda ) \ra \Lambda^{\rm op} \otimes \Lambda\ra
\Lambda^{\rm op} /(\overline{\Delta} ) \otimes \Lambda/(\Delta ) \ra 0 .$$
%where $D$ is $\Delta \Lambda \otimes \Lambda \oplus \Lambda \otimes \Delta \Lambda $.
Noticing $\overline{\Delta}= \Delta$, we have
\begin{equation}\label{oo3} 0 \ra\frac{(\Delta \Lambda^{\rm op} \otimes \Lambda )\oplus( \Lambda^{\rm op} \otimes \Delta \Lambda)}{ \Delta \Lambda^{\rm op} \otimes \Delta \Lambda } \lra \frac{ \Lambda^{\rm op} \otimes \Lambda }{ \Delta (\Lambda^{\rm op} \otimes \Lambda ) } \lra \frac{\Lambda^{\rm op}}{(\Delta )} \otimes \frac{\Lambda}{(\Delta )}\ra 0 \ \ \ (\mathrm{exact}).
\end{equation}
%\begin{equation}\label{oo3} 0 \ra\frac{(\Delta \Lambda \otimes \Lambda^{\rm op} )\oplus( \Lambda \otimes \Delta \Lambda^{\rm op} )}{ \Delta \Lambda \otimes \Delta \Lambda^{\rm op} } \lra \frac{ \Lambda \otimes \Lambda^{\rm op} }{ \Delta (\Lambda \otimes \Lambda^{\rm op} ) } \lra \Lambda /(\Delta ) \otimes \Lambda^{\rm op} /(\Delta )\ra 0 \ \ \ (\mathrm{exact}).
Denote by $\gamma$ the associated Bockstein map.
%$ \overline{\psi} : (\Delta \Lambda \otimes \Lambda^{\rm op} )\oplus( \Lambda \otimes \Delta \Lambda^{\rm op} ) \ra \Z[t^{\pm 1}]/ ( \Delta) $ by $ \psi_0 (\Delta x\otimes y \oplus z\otimes \Delta w)=\bar{x}y + \bar{z}w $,
Then, using \eqref{oo38}, they induce
$$(\psi_l \oplus \psi_r )_* : H^3 ( E_K, \partial E_K ; \frac{(\Delta \Lambda \otimes \Lambda^{\rm op} )\oplus( \Lambda^{\rm op} \otimes \Delta \Lambda) }{ \Delta \Lambda^{\rm op} \otimes \Delta \Lambda }) \ra H^3 ( E_K, \partial E_K ; \frac{ \Lambda }{( \Delta ) }) . $$
Moreover, let us define the following composite homomorphisms:
\begin{equation}\label{oo389} \Upsilon : H^2( E_K, \partial E_K ; \Lambda^{\rm op}/ (\Delta
) \otimes \Lambda/ (\Delta
) ) \xrightarrow{ \ (\psi_l \oplus \psi_r )_* \circ \gamma \ }
H^3 ( E_K, \partial E_K ; A) \xrightarrow{ \  \bullet \cap [ E_K, \partial E_K]  \ } A, \end{equation}
\[ \Phi : H^2( E_K, \partial E_K ; \Lambda/ (\Delta
) \otimes \Lambda^{\rm op}/ (\Delta
) ) \xrightarrow{\  \bullet \cap \Sigma \ }
\Lambda^{\rm op}/ (\Delta
) \otimes \Lambda/ (\Delta
) \xrightarrow{ \ \psi_0 \ } A, \]
where $\bullet \cap   [ E_K, \partial E_K] $ is the cap-product with the relative fundamental 3-class in $H_3(E_K, \partial E_K ;\Z) $.

In addition, we will explain a Leibniz rule of Bockstein maps, and show Lemma \ref{bb1l8} below.
%Recall the exact sequences \eqref{oo1}, \eqref{oo2}, \eqref{oo3}. % in \S \ref{S1}.
%\begin{equation}\label{oo3} 0 \ra\frac{(\Delta \Lambda \otimes \Lambda^{\rm op} )\oplus( \Lambda \otimes \Delta \Lambda^{\rm op} )}{ \Delta \Lambda \otimes \Delta \Lambda^{\rm op} } \lra \frac{ \Lambda \otimes \Lambda^{\rm op} }{ \Delta (\Lambda \otimes \Lambda^{\rm op} ) } \lra \Lambda /(\Delta ) \otimes \Lambda^{\rm op} /(\Delta )\ra 0 \ \ \ (\mathrm{exact}).
Let $\alpha, \beta, \gamma$ be the associated Bockstein maps with \eqref{oo1}, \eqref{oo2}, \eqref{oo3}, respectively. Let $C $ be the first term in \eqref{oo3}, and $\nu: (\Delta \Lambda^{\rm op} \otimes \Lambda )\oplus( \Lambda^{\rm op} \otimes \Delta \Lambda^{\rm op} )\ra C $ be the projection.
Then, it follows from \cite[Proposition in p. 451]{S} that
\begin{equation}\label{oo43} \gamma ( u \smile v)= \nu_*( \alpha( u) \smile v +(-1)^{\mathrm{dim}(u)} ( u \smile \beta( v) )) ,
\end{equation}
for $u \in H^*(X;\Lambda /(\Delta ) ), v \in H^*(X;\Lambda^{\rm op} /(\Delta ) ) .$

% \[ 0 \ra \Lambda \stackrel{\Delta}{\lra} \Lambda \lra \Lambda /(\Delta )\ra0 \]
\begin{lem}\label{bb1l8} Let $\alpha,\beta, \gamma$ be as above.
Take $ u,v \in H^1(E_K , \partial E_K ; \Lambda/\Delta) $. Then,
\begin{equation}\label{oo4344} \langle \gamma (u \smile v), [E_K, \partial E_K ] \rangle = \frac{1+t}{1-t} \langle \beta u \smile  v, [E_K, \partial E_K ] \rangle.
\end{equation}
\end{lem}
\begin{proof}
%Since $ \mathrm{Bl}_K$ is hermitian,
From the definitions of $ \psi_{r}$ and $ \psi_{l}$, we notice that
\begin{equation}\label{oo9} %\psi_{l} ( (tV' -V ) \cdot \mathfrak{s}(u) ,\mathfrak{s}(v) ) =
\psi_{r} (u, (tV -V ') \cdot \mathfrak{s}(v) ) = %-t^{-1} \overline{ \psi_{r} (\mathfrak{s}(v), (tV -V ') \cdot \mathfrak{s}(u) ) }=
\psi_{l} ( \overline{(t V -V ')' } \cdot \mathfrak{s}(u),v).
\end{equation}
By \eqref{oo43}, the left hand side in \eqref{oo4344} is formulated
\[\psi_{l} ( (tV -V') \mathfrak{s}(u) , v) -\psi_{r} \bigl( u, (tV -V' ) \cdot \mathfrak{s} (v) \bigr) , \]
which is computed as
\[ = \psi_{l} ( (tV -V') \mathfrak{s}(u) , v) - \psi_{l} ( (t^{-1} V' -V ) \cdot \mathfrak{s}(u),v )\]
\[ = \psi_{l} ( (tV -V') \mathfrak{s}(u) , \mathfrak{s}(v)) + \psi_{l} ( (V- t^{-1} V' ) \cdot \mathfrak{s}(u),v )\]
\[ = \psi_{l} ( (tV + V -V'-t^{-1} V') \mathfrak{s}(u) , v) \]
%\[ = \psi_{l} ( \overline{\mathfrak{s}(u)} \cdot (t^{-1}V -V') \bigr) , v') +t^{-1} \psi_{l} ( \overline{\mathfrak{s}(u) }\cdot (t V -V') , v')\]
\[ = (1+t) \psi_{l} (( tV -V') \mathfrak{s}(u) , v ) .\]
%?????which turns out to be $(1+t) (\overline{ \mathfrak{s}(u)} (V -t^{-1}V')) v' $.
Since $\langle \beta u \smile v, [E_K, \partial E_K ] \rangle$ is
$(1-t) \psi_{l} (( tV -V') \mathfrak{s}(u) , \mathfrak{s}(v) )$ by Proposition \ref{bbbn2}, we have the desired equality.
\end{proof}

\subsection{Final discussion}\label{kkl4}
To prove Theorem \ref{clAl22}, we need a lemma: 
\begin{lem}\label{bb1l98} There is a constant $\alpha_K$
such that $\Phi = \alpha_K \Upsilon$, where $ \Phi $ and $\Upsilon$ are given in \eqref{oo389}. 
\end{lem}
\begin{proof}First, we notice that, if the tensor is defined over $\Lambda$, 
the coefficient $ \Lambda^{\rm op}/ (\Delta
) \otimes_{\Lambda} \Lambda/ (\Delta 
) $ is a trivial coefficient, and is additively isomorphic to $\Lambda/ (\Delta 
) $. Thus, we have an additive isomorphism 
$$ H^2( E_K, \partial E_K ; \Lambda^{\rm op}/ (\Delta
) \otimes_{\Lambda} \Lambda/ (\Delta
) ) \cong  H^2( E_K, \partial E_K ; \Z) \otimes \Lambda/ (\Delta
) \cong \Lambda / (\Delta
) .$$
By the definition of $  \mathcal{Q}_{\psi}$ (see \eqref{kiso}), 
$  \mathcal{Q}_{\psi}$ factors through this second homology.
Thus, the maps $ \Phi$ and $ \Upsilon$ are regarded as multiplications 
of $P $ and $Q $ for some $P , Q \in \Lambda/ (\Delta
)$, respectively.
Here, we notice that $P$ is invertible in $\Lambda/ (\Delta
) $, since, if not so, the Blanchfield pairing is not non-singular. 
Hence, defining $\alpha_K$ by $P^{-1} Q$, we have 
$ \Phi = \alpha_K \Upsilon$ as required.
\end{proof}

%\begin{lem}\label{1234} a
%There is an integer $a_K \in \Z $ forwhich
%Then the following equality holds: $$i_*(\tilde{\Sigma}) \cap ( x \smile y)= \mathrm{Int}( \xi(x), \xi(y))\in \mathbb{Z}, \ \ \ \ \ \mathrm{for \ any \ \ }x,y \in H^1(\widetilde{Y}_K, \partial \widetilde{Y}_K;\Z ). $$
%\end{lem}
%\begin{proof} Since $ i_*(\tilde{\Sigma} ) = \tilde{\mu} $ is shown \cite[Lemma 6.1]{Kaw}, we have
%\end{proof}
%Then \cite[Lemma 1.5]{Kaw1} showed the equality
%\begin{lem}[{\cite[Lemma 1.5]{Kaw1}}]\label{lem321e}
%For any $x,y \in H^1(\widetilde{Y}_K, \partial \widetilde{Y}_K;\Q )$, we have
%\begin{equation}\label{ooqq}
%\end{lem}
\begin{proof}[Proof of Theorem \ref{clAl22}]
Recall the definition of $ \Upsilon$ in \eqref{oo389},
the left hand side in \eqref{oo4344} is equal to $\Upsilon( u \smile v) $.
If $u= \kappa(x) , v=\kappa(y)$, the right hand side in \eqref{oo4344}
equals $\frac{1+t}{1-t} \cdot \mathrm{Bl}_K(x,y)$ by Corollary \ref{bb1787}.
On the other hand, by
the definition of $ \Phi $ in \eqref{oo389},
$\Phi( u \smile v) = \mathcal{Q}_{\psi_0} (u,v)$.
By Lemma \ref{bb1l98}, $ \Phi = \alpha_K \Upsilon$ for some $\alpha_K \in \Lambda/(\Delta)$;
we have
$$ \mathcal{Q}_{\psi_0 } (\kappa(x) ,\kappa(y))= \mathcal{Q}_{\psi_0} (u,v)=\Phi( u \smile v)= \alpha_K\Upsilon( u \smile v)
$$
$$= \alpha_K \frac{1+t}{1-t} \cdot \mathrm{cBl}_K(u,v) = \alpha_K \frac{1+t}{1-t} \cdot \mathrm{Bl}_K(x,y) \in \Z[t^{\pm 1}]/ ( \Delta) ,$$
which is the required equality. %conclusion:
It remains to show $\overline{\alpha_K}= \alpha_K.$ Indeed, 
since $\mathrm{Bl}_K $ is hermitian and $\mathcal{Q}_{\psi_0} $ is anti-hermitian, 
$ \alpha_K$ must satisfy $\overline{\alpha_K}= \alpha_K.$
% as required.
% (resp. $\Phi( u \smile v) $). ?????? We now suppose
\end{proof}

\section{Computation I: small knots and some non-fibered knots}\label{S51}
From this section, we will give the resulting computations of
the Blanchfield pairings and the other pairing $\mathcal{Q}_{\psi_0}$
for some knots; recall the definition $\mathcal{Q}_{\psi_0}$ in Theorem \ref{mainthm1}.
Here, the former pairing is easily computed by Proposition \ref{bbbn2} in terms of Seifert matrices, and
the latter is computed from Theorem \ref{mainthm1} in terms of $X$-colorings.
Here, we use data of the Seifert matrices from KnotInfo \cite{CL}.

%In this section,
We give a list of
the resulting computations of all knots of crossing number $<8$; see Table \ref{tt1}.
Here, if $\alpha_K \neq 1$ in the table, it is not hard to verify that $\alpha_K$ is
not invertible and not an zero divisor in $ \Lambda/\Delta$.

\begin{table}[hbtp]
\centering
\begin{tabular}{ccc}
\hline
Knot type & Alexander polynomial $\Delta$ & $\alpha_K \in \Lambda /(\Delta)$ \\
\hline \hline
$3_1$ & $t^2-t+1 $ & 1 \\
$4_1$ & $t^2-3t+1 $ &1 \\
$5_1$ & $t^4-t^3+t^2-t+1 $ & $t^{-1} +2+ t $ \\
$5_2$ & $2t^2-3t+2 $ & 1 \\
$6_1$ & $2t^2-5t+2 $ & 1 \\
$6_2$ & $t^4-3t^3+3t^2-3t+1 $ & $3 t^{-1}- 7 + 3 t$ \\
$6_3$ & $t^4-3t^3+5t^2-3t+1 $ & $t+t^{-1}$ \\
$7_1$ & $t^6-t^5+t^4-t^3+t^2-t+1 $ & $3t^{-2} - 2 t^{-1} + 4 - 2 t + 3 t^2$ \\
$7_2$ & $3t^2-5t+3 $ & $ 2 t^{-1}- 3 + 2 t$ \\
$7_3$ & $2t^4-3t^3+3t^2-3t+2 $ & $(-3 + 2 t) (-2 + 3 t^{-1})$ \\
$7_4$ & $4t^2-7t+4 $ & 1 \\
$7_5$ & $2t^4-4t^3+5t^2-4t+2 $ & $2(2t^{-1} - 3 + 2t)$ \\
$7_6$ & $t^4-5t^3+7t^2-5t+1 $ & $t^{-1} - 5 + t $\\
$7_7$ & $t^4-5t^3+9t^2-5t+1 $ & $t^{-1} - 4 + t $ \\
\hline
\end{tabular}
\caption{The constants $\alpha_K$ for knots of crossing number $<8$. }
\label{tt1}
\end{table}
%In the lists,

Similarly, 
% Furthermore, 
we can compute $\mathcal{Q}_{\psi_0}$ for knots of higher crossing number.
In our experience in the computation, when the Alexander polynomial $\Delta$ is divisible by
a square of a polynomial and $K$ is not fibered, the constant $\alpha_K$ might be a non-trivial divisor of $\Delta$ and not a unit. 
As examples, we give a table:
\begin{table}[hbtp]
\centering
\begin{tabular}{ccc}
\hline
Knot type & Alexander polynomial $\Delta$ & $\alpha_K \in \Lambda /(\Delta)$ \\
\hline \hline
$8_{20}$ & $(t^2-t+1 )^2$ & $t-1+t^{-1}$ \\
$11_{73}$ & $(t^2-t+1 )^2$ & 0 \\
$12_{a0169 }$ & $(2 t^2 - 3 t + 2)^2 $ &$4 (-2 + t) (2 t^{-1} - 3 + 2 t)$\\
$12_{n0057 }$ & $(t^2-t+1 )^2 $ &$(t + t^{-1}) (t - 1 + t^{-1})$\\
$12_{n087 }$ & $(2 t^2 - 3 t + 2)^2 $ &$2t - 3 + 2 t^{-1}$\\
$12_{n0279 }$ & $(t^2-3t+1 )^2 $ & $(t-3+t^{-1}) (3t-7+3t^{-1})$ \\
% $12_{n0394 }$ & $(t^2-3t+1 )^2 $ & $ $ \\
\hline
\end{tabular}
\caption{The constants $\alpha_K$ for some non-fibered knots. }
\label{table:data_type}
\end{table}
%In these cases, this tables says that 
Thus, the Blanchfield pairings of such knots can not be recovered from $ \mathcal{Q}_{\psi_0} $'s. 
As these tables imply, it might be a difficult problem to
give a formula to determine $\alpha_K$ for every knot $K$.

\section{Computation II: Pretzel knots}\label{S52}
We will focus on the Pretzel knot as in Figure \ref{ftf}.
Take odd numbers $p,q,r \in \Z$ such that $ p=2\ell +1, q= 2m+1, r=2n+1$.
Then, the Alexander polynomial is known to be
$\Delta= \frac{1}{4} \bigl( ( pq + q r + rp)(t^2-t+1)+t^2+t+1\bigr) $.
By observing the discriminant, $\Delta$ can not be any square of some polynomial. %We will show
The purpose of this section is to show the following: %at
\begin{thm}If $K$ is the Pretzel knot $P(p,q,r)$, then $K$ satisfies the assumption of Theorem \ref{clAl22}, and $\alpha_K=1$.
\end{thm}
\noindent
In other words, $ \mathrm{Bl}_K $ is completely recovered from the cohomology pairing.
Since this theorem immediately follows from a comparison between Propositions \ref{kk6} and \ref{kk7} below, we will show the propositions.

\

\

\begin{figure}[htpb]
\begin{center}
\begin{picture}(10,60)
%\put(-128,42){\LARGE $D $}
\put(-159,61){\large $\alpha^u $}
\put(-99,61){\large $\beta^u $}
\put(-55,61){\large $\gamma^u $}
\put(-159,-13){\large $\alpha_b $}
\put(-98,-15){\large $\beta_b$}
\put(-60,-13){\large $\gamma_b $}

% \put(88,64){\large $\alpha_1 $}
\put(88,51){\large $\alpha_1 $}
\put(87,-2){\large $\alpha_m $}
\put(88,24){\large $\alpha_i $}
\put(91,10){\normalsize $\vdots $}
\put(91,38){\normalsize $\vdots $}
%\put(88,24){\large $\alpha_i $}

%\put(-18,42){\LARGE $\CC $}
%\put(53,42){\LARGE $\Gamma_{\CC} $}
%\put(-161,22){\pc{coloringtrefoil4}{0.53104}}
%\put(-186,22){\pc{pic12b}{0.402104}}
%\put(-46,22){\pc{12a169.knot}{0.25702104}}
\put(-351,26){\pc{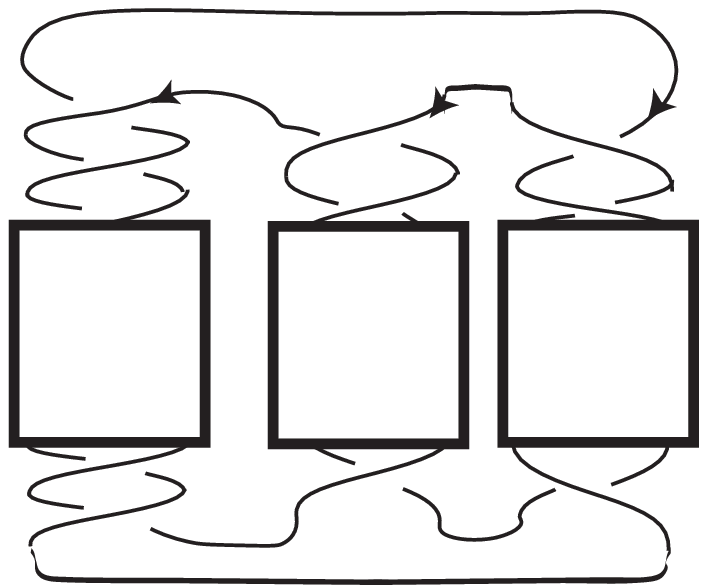}{0.58104}}
\put(96,26){\pc{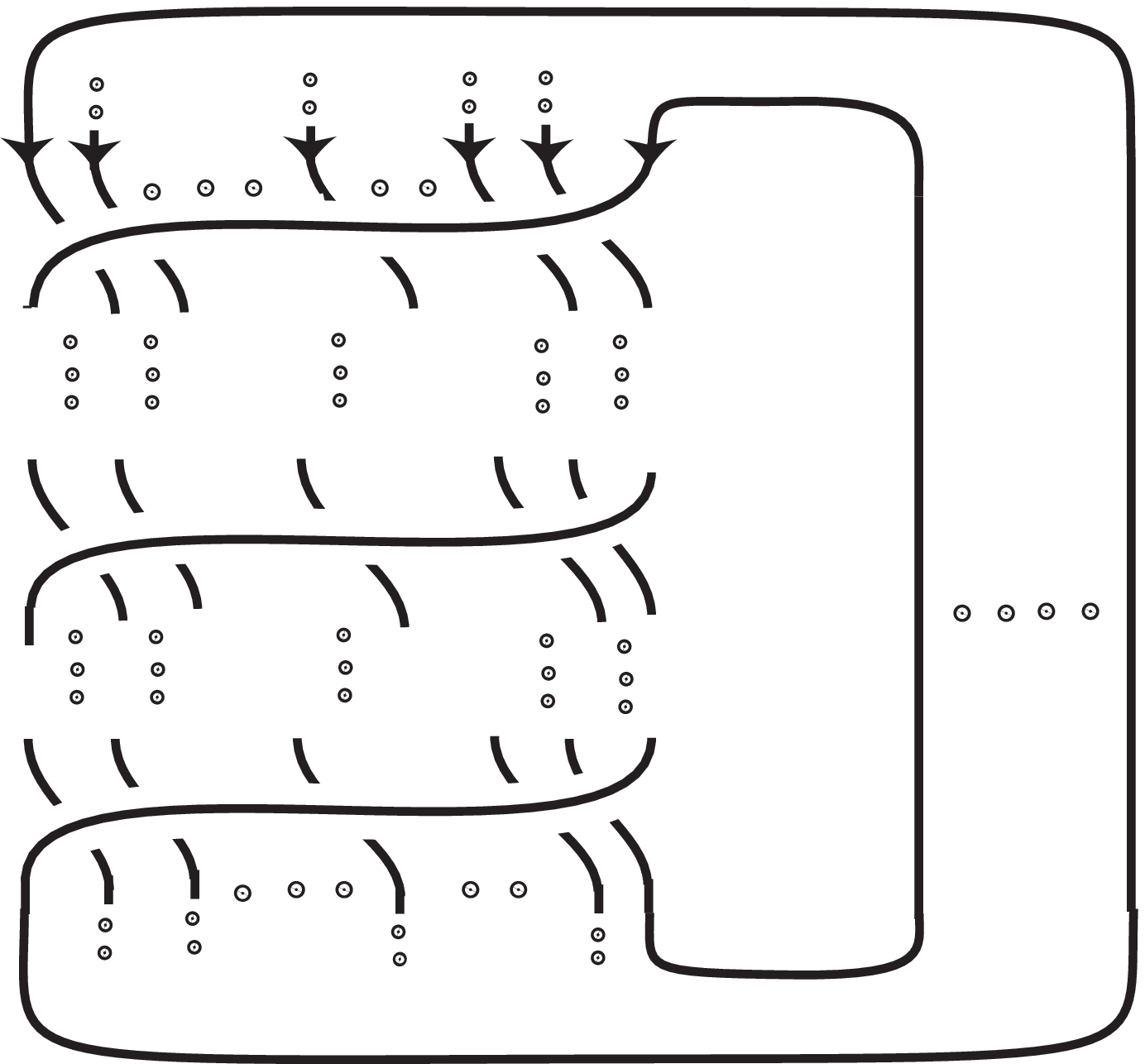}{0.2473104}}
\end{picture}

\

\end{center}\caption{\label{ftf} The Pretzel knot $P(p,q,r)$ and the $T_{m,n}$-torus link with labeled arcs. Here, the boxes in the left hand side mean $p$-, $q$-, $q$-twists, respectively.}
\end{figure}

% To show this theorem,
For the purpose,
one computes the Blanchfield pairing $\mathrm{Bl}_{K} $.
According to \cite[Example 6.9]{Lic},
we can choose the Seifert matrix of the form $V= \frac{1}{2}\left(
\begin{array}{cc}
p+q & q +1 \\
q-1 & q+r
\end{array}
\right)$.
\begin{prop}\label{kk6}
The kernel $\Ker_{\Lambda/\Delta} (tV- V') \subset ( \Lambda/\Delta )^2 $ is generated by two elements
$$((1 + m + n) ( t-1), t + m t-m) , \ \ \ \ w= ( m t -1 - m , (1 + \ell + m) (t-1 )) . $$
Furthermore, we have
$$ \left(
\begin{array}{cc}
\mathrm{cBl}_K( v,v) & \mathrm{cBl}_K( v,w) \\
\mathrm{cBl}_K( w,v) & \mathrm{cBl}_K( w,w)
\end{array}
\right) = (1-t^{-1})\left(
\begin{array}{cc}
(1-t) (1 + m + n)& (-1 - m + m t) \\
(-m + t + m t)&(1-t) (1 + m + \ell)
\end{array}
\right) . $$
\end{prop}
Thanks to Proposition \ref{bbbn2}, the proof can be easily obtained by the help of a computer program of mathematica; we omit the details.

On the other hand, we will compute the cohomology pairing $\mathcal{Q}_{\psi_0}$.
\begin{prop}\label{kk7}
Let $X$ and $A$ be $\Lambda/\Delta $ in usual.
Consider the submodule, $\mathcal{K} $, of $ ( \Lambda/\Delta )^2 $ independently generated by two elements
$$v=( (1 + m + n) ( t-1), t + m t-m) , \ \ \ \ w= ( m t - 1-m, (1 + \ell + m) (t-1)) \in ( \Lambda/\Delta )^2. $$
Then, there exists a $\Lambda$-isomorphism $ \theta: \mathcal{K} \cong \mathrm{Col}_X^{\rm red}( D)$
such that
\[ \left(
\begin{array}{cc}
\mathcal{Q}_{\psi_0}( \theta(v),\theta(v)) & \mathcal{Q}_{\psi_0}( \theta(v),\theta(w)) \\
\mathcal{Q}_{\psi_0}( \theta(w),\theta(v)) & \mathcal{Q}_{\psi_0}( \theta(w),\theta(w) )
\end{array}
\right) = (1+t^{-1})\left(
\begin{array}{cc}
(1-t) (1 + m + n)& (-1 - m + m t) \\
(-m + t + m t)&(1-t) (1 + m + \ell)
\end{array}
\right). \]
\end{prop}

\begin{proof}
For simplicity, we suppose that all of $ p,q,r$ is positive. Since the proofs of other cases can be done in the same way, we omit considering the other cases.

We first notice the following lemma, which can be obtained by definitions:
\begin{lem}\label{l1} Consider the $(2,2)$-tangle with $(2N+1)$-twist, and the labels of the arcs in Figure \ref{ftf2}. Choose $a_1,a_2,x_1,x_2,y_1,y_2 \in X$.
For $j \in \{ 1,2\}$, the assignment $\mathcal{C}_j$
$$ \alpha_k  \longmapsto a_j + k(1-t) (y_j- x_j )+ x_j, \ \ \ \ \ \ \beta_k \longmapsto a_j + k(1-t) ( y_j- x_j )+ y_j $$
defines an $X$-coloring. Moreover,
the weight sum with respect to the $2N+1$ crossings is
\begin{equation}\label{1306}
\mathcal{Q}_{\psi_0}( \mathcal{C}_1, \mathcal{C}_2 )=- \bigl( (\overline{x_1-y_1})(a_2+y_2) + N (\overline{x_1-y_1})(x_2- y_2)\bigr) (1-t^{-1}) .\end{equation}
\end{lem}
\begin{figure}[htpb]
\begin{center}
\begin{picture}(10,60)
%\put(-128,42){\LARGE $D $}
\put(-85,41){\large $\alpha_0 $}
\put(-85,13){\large $\beta_0 $}
\put(-50,46){\large $\beta_1$}
\put(-50,3){\large $ \alpha_1 $}

%\put(0,13){\Large $ (2N-1) $-twist}

% \put(88,64){\large $\alpha_1 $}
\put(95,45){\large $\beta_{N+1} $}
\put(95,4){\large $\alpha_{N+1} $}
%\put(88,24){\large $\alpha_i $}

%\put(-18,42){\LARGE $\CC $}
%\put(53,42){\LARGE $\Gamma_{\CC} $}
%\put(-161,22){\pc{coloringtrefoil4}{0.53104}}
%\put(-186,22){\pc{pic12b}{0.402104}}
%\put(-46,22){\pc{12a169.knot}{0.25702104}}
\put(-281,26){\pc{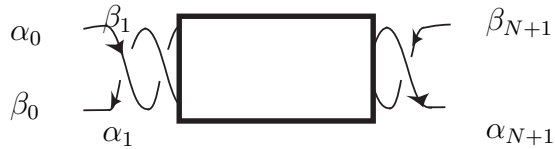}{0.68104}}

\end{picture}
\end{center}\caption{\label{ftf2} The $(2,2)$-tangle as a $(2N+1)$-twist. Here the box means a twist.}
\end{figure}

%This lemma can be obta
Given an $X$-coloring $\mathcal{C}$ of $P(p,q,r)$,
we put $a,x,y \in X$ such that $ \mathcal{C} (\alpha^u)=a, \ \mathcal{C} (\beta^u)=a+x, \ \mathcal{C} (\gamma^u)=a+y.$ By Lemma \ref{l1}, we have the simultaneous equations
\[\mathcal{C} (\alpha_b)=(\ell+1)(1 - t ) x+a = m (y-x)(1 - t) +y +a,\]
\[ \mathcal{C} (\beta_b)=(1 + m) (y-x)(-1 + t) + x+ a = -n y (1 - t) +a, \]
\[ \mathcal{C} (\gamma_b)=-(1 + n) y (-1 + t) +y +a = \ell (1 - t)x+x+a \in \Lambda/\Delta. \]
Then, by the help of a computer program of Mathematica, we have two solutions
\[ \begin{cases}
x_1 &= (1 + m + n) ( t-1),\\
y_1 & = t + m t-m,
\end{cases} \ \ \ \ \ \mathrm{or} \ \ \ \ \ \
\begin{cases}
x_2 &= m t - 1-m,\\
y_2 &=(1 + \ell + m) (t-1).
\end{cases} \]
Furthermore, it can be verified that every solutions of $(x,y) $ is
a linear sum of the two solutions.
Hence, we have the desired isomorphism $ \theta: \mathcal{K} \cong \mathrm{Col}_X^{\rm red}( D)$.

For $i \in \{ 1,2\}$, let $ \mathcal{C}_i$ be the $X$-coloring associated with the solution $(x_i, y_i)$.
Thanks to Lemma \ref{l1} again,
given two $X$-colorings $ \mathcal{C}_i$ and $ \mathcal{C}_j'$,
the sum $\mathcal{Q}_{\psi_0} (\mathcal{C}_i, \mathcal{C}_j' )$ is equal to
\[ -\bigl( \bar{x_i}(a_j+x_j) + \ell \bar{x_i} x_j + (\bar{x}_i-\bar{y}_i)(a_j+y_j) + m (\bar{x}_i-\bar{y}_i)(x_j- y_j)+ \bar{y}_i a_j + n \bar{y}_i x_j \bigr) (1-t^{-1}) . \]
Using Mathematica for the computation modulo $\Delta$,
we can obtain the desired equality in the $2 \times 2$-matrix.
\end{proof}

\section{Computation III: the torus knot}\label{S53}

%\vskip 1.2228888419937pc

%We assume that $ m$ and $n $ are relatively prime or are equal.

%Let $L$ be the $(m,n)$-torus link $T_{m,n}$,
%in the case where
We will compute the cohomology
%Thanks to Theorem \ref{clAl22}, we will address the first to determine the Blanchfield
pairing of the $ (m,n )$-torus knot $T_{m,n}$.
Here note the known fact that
the Alexander module is isomorphic to $ \Z [t^{\pm 1}]/( \Delta) $,
where
the Alexander polynomial $ \Delta$ is $ (t^{nm}-1)(t-1)/ \bigl( (t^n-1 )(t^m-1)\bigr)$; see \cite{Rol}.
\begin{prop}\label{exact}
Fix %four integers
$(n,m ,a,b) \in \mathbb{Z}^4$ with $an + bm=1$, and let $K=T_{m,n}.$
%As in \S \ref{}, let $f: \pi_L \ra G=\Z$ be the abelianization and,
Then,
\begin{equation}\label{176} \mathcal{Q}_{\psi_0} (y_1, y_2 ) = \frac{nm (1-t^{-1} )}{(1- t^{ bm })( 1- t^{ an }) } \cdot \bar{y}_1 y_2 \in \Z[t^{\pm 1}]/ (\Delta) , \end{equation}
%\ \ \ \ \
for $ y_1, \ y_2 \in H^1 (S^3 \setminus T_{m,n} ;\Z [t^{\pm 1}]/( \Delta)) \cong \Z [t^{\pm 1}]/( \Delta). $
% In particular, if $n$ and $m$ are odd, the value is equivalent to the classical Blanchfield pairing. $\mathrm{Bl}_K$.
\end{prop}
Remark that the coefficient in \eqref{176} lies in $\Z[t^{\pm 1}]/(\Delta ) $ because of
l'H\^{o}pital's rule in $\Z[t^{\pm 1}]$.
% this formula is of rank $ g_*$.
%We easily see
\begin{proof}Let $ X= A= \Z [t^{\pm 1}]/( \Delta)$.
By Theorem \ref{clAl22}, we have $ \mathrm{Col}_X(D) \cong H_1 (S^3 \setminus K ; A ) \oplus X \cong X ^2$;
it is enough to compute the bilinear form $\mathcal{Q}_{\psi } $ with
$\psi (y, z)= \bar{y} z $.

To this end, let us start by examining $ \mathrm{Col}_X(D)$ in more details.
Let $\alpha_1, \dots, \alpha_m $ be the arcs depicted in Figure \ref{ftf}.
Because of the shape of $D$, every coloring in $ \mathrm{Col}_X(D)$ is characterized by
colors of these $m $ arcs.
%$\alpha_1, \dots, \alpha_m $;
Hence, we can view $ \mathrm{Col}_X(D)$ as a submodule of $X^m $.
In addition, for $k \in \{ 1,2 \}$, consider elements of the forms
\begin{equation}\label{1122}\vec{y}_{k}= ( \delta_k , \ y_k + \delta_k , \ \frac{1- t^{2 an }}{1- t^{ an }}y_k+ \delta_k , \dots , \ \frac{1- t^{ an (m -1 ) }}{1- t^{ an }}y_k+
\delta_k ) \in X^m, \end{equation}
% where $$a_j := \bigl(a (1- \zeta^{ \beta j})(1-t)/(1 - \zeta) + a (1- \zeta^{ \beta j}) + \delta \bigr) $$
for some $y_k , \ \delta_k \in X.$ From this view, we can easily see that these elements \eqref{1122} define $X$-colorings.
Further, the first and second components imply that these elements in \eqref{1122} give a basis of $ \mathrm{Col}_X(D) \cong X^2$.
Let $\tau_{i,j}$ be the $j$-th crossing point on the arc $\alpha_i$.
Then, concerning the $X$-coloring $\mathcal{C}$ arising from \eqref{1122}, the colors around $\tau_{i,j}$ as Figure \ref{koutenpn} are formulated as
$$ \bigl( \mathcal{C} (\alpha_{\tau_{i,j}}) , \ \mathcal{C} ( \beta_{\tau_{i,j}}) \bigr)
= \bigl( \frac{1- t^{ bm j }}{1- t^{bm}} t^{an(i- 2)} y_k + \frac{1- t^{ an (i-1) }}{1- t^{an}}y_k + \delta_k , \
\frac{1- t^{ an (i-1) }}{1- t^{an}}y_k + \delta_k \bigr) \in A^2 . $$

Accordingly, we now deal with the 2-form $\mathcal{Q}_{\psi} (\vec{y}_1,\vec{y}_2 )$.
By definition, compute it as
%is, by definition, formulated as
\begin{align*}
\lefteqn{} &
\sum_{i \leq m, \ j \leq n-1} \psi_0 \bigl( \ t^{ an(i - 2) } \frac{1- t^{ bmj }}{1- t^{ bm }} y_1
, \ ( \frac{1- t^{ an(i-1) }}{1- t^{an }}y_2 + \delta_2 ) (1-t^{-1}) \ \bigr)
\\
&= \psi_0  \bigl( \ \sum_{j =1 }^{n-1} \frac{1- t^{ bmj }}{1- t^{ bm }} y_1 ,\ \ ( \sum_{i =1}^{ m } \frac{t^{ an } - t^{ an(2-i) }}{1- t^{an }} ) y_2 (1-t^{-1}) \bigr)
\\
&= \psi_0 \bigl( \frac{ 1- t^{ bm(n-1) } -(n-1)(1- t^{ bm} ) }{ (1- t^{ bm })^2} y_1 , \
\frac{ m(1- t^{ an} ) + 1- t^{- amn} }{ (1- t^{ an })^2} t^{an} \cdot y_2 (1-t^{-1}) \bigr)\\
&
=\psi_0  \bigl(\frac{n}{1- t^{ bm }} \cdot y_1, \ t^{ an }\frac{m}{1- t^{ an } }\cdot y_2 (1-t^{-1}) \bigr)
= \frac{nm(1-t^{-1} ) }{(1- t^{ bm })( 1- t^{ an }) } \cdot \bar{y}_1 y_2 \in \Z [t^{\pm 1}]/(\Delta ),
\end{align*}
where the first equality is obtained by the $t$-invariance of $\psi$ and
the equality $ \sum_{i =1}^{ m} t^{an i} =0 $, and
an elementary computation lifted in $\Z[t^{\pm 1}]$ can imply
the third equality by noting $t^{mn}=1 \in X $.
% by definitions, the desired one as bilinear forms over $A$.
%The last term is $(1-t)^{-1}(1+t) \mathrm{Bl}_K(y_1,y_2 )$ by Theorem \ref{clAl22}.
%So, notice that, $\Delta_{K}(-1 )= 1$ if $ nm$ is odd, and $\Delta_{K}(-1 )= m$ if $ n$ is even; we may divide the value $ \mathcal{Q}_{\psi} (\vec{y}_1,\vec{y}_2 )$ by $ 1+t $.
%In the sequel, since $(1-t)^{-1} \in A$, the form $ \mathcal{Q}_{\psi} (\vec{y}_1,\vec{y}_2 )$ is reduced to the required formula.
%Since , we obtain immediately the desired formula from this computation.
%, and the second one is done by .
%Noticing , the last term containg in $\delta$ vanishes.Therefore, xince $\psi$ is , it is reduced to
\end{proof}

Finally, we will give a comparison with the Blanchfield pairing of $T_{m,n }. $
The pairing has not since been computed; the essential reason is
the Seifert genus is $ (n-1)(m-1)/2$, i.e.,
it seems impossible to compute the pairing from Theorem \ref{bbbn2}
%the previous formula \cite[Proposition 2.5]{T2}
using the Seifert matrix.
%we can also formulate the Blanchfield pairing as a matrix whose rank is
%????since
Furthermore, it is a subtle problem whether ${\rm Bl}_{T_{m,n } } $ can be
recovered from $ \mathcal{Q}_{\psi_0}$ or not.
In fact, the coefficients $nm / (1- t^{ bm })( 1- t^{ an }) $ are not units in $\Lambda/(\Delta)$ in many cases; for example, we can easily verify that, if $m$ is even, the coefficient is divisible by $(1+t)^2$, and that if $(m,n)=(3,7)$, $ \Delta(2)= 7\cdot 337 $ and the coefficient modulo $t+2$ is $2^3\cdot 3^2 \cdot 5 \cdot 7^2 \cdot 17. $
Meanwhile, the coefficient has some important information of ${\rm Bl}_{T_{m,n } } $:
for example, Matumoto \cite{Mat}, Kearton \cite{Kea2}, and Litherland \cite{Lit} independently use other technical formula to compute all the local signatures of $T_{m,n}$, and the signatures can be recovered from the form of the coefficient.
% However, our result is still simpler as follows:

\appendix
\section{Universality of Alexander quandle cocycle invariants.}\label{999}
The paper \cite{CJKLS} constructed a knot invariant from quandle cocycles.
However, the invariant was defined in a combinatorial way without topological meanings.
This appendix reviews the invariant, and gives a topological meaning in the Alexander case, as an application of Theorems \ref{clAl22} and \ref{mainthm1}.

For the purpose, we first briefly review the invariant,
supposing that the reader has read Sections 2--4. %& \ref{s3s3w}.
%Let $K$ be a knot.
% Denote the abelianization $\pi_1(\SK) \ra \Z$ of a knot $K$ by $f$.
Let $X$ be a $\Z[t^{\pm 1}]$-module, which is regarded as an Alexander quandle.
%Then, the pair of $X$ and the binary operation $(x,y) \mapsto tx +y-ty$ is called the Alexander quandle; see \cite{CJKLS}.
Further, given an abelian group $A$, we suppose a map $\phi: X^2 \ra A$ satisfying the equality
%be a 2-cocycle, that is,
\begin{equation}\label{kis22o} \phi(x , z )- \phi(y , z ) - \phi( x \lhd y , z )+ \phi( x \lhd z , y \lhd z) =0, \end{equation}
\[ \phi (x,x)=0, \]
for any $x,y,z \in X$.
Such a map $\phi$ is called {\it a quandle 2-cocycle} \cite{CJKLS}.
%Further, for any $X$-coloring $\mathcal{C}$, we consider
Then, in analogy of \S \ref{Ss3}, %we can take a coloring set $\mathrm{Col}_X (D) $,
let us define a map
$$\mathcal{I}_{\Phi} : \mathrm{Col}_X (D) \lra A; \ \ \ \ \mathcal{C} \longmapsto \sum_{\tau} \epsilon_{\tau} \cdot \phi( \mathcal{C}(\alpha_{\tau}), \ \mathcal{C}(\beta_{\tau})),$$
where $\tau $ ranges over all the crossings of $D$, and
the symbols $\alpha_{\tau} , \ \beta_{\tau} $ are the arcs and $\epsilon_{\tau} \in \{ \pm 1\}$ is the sign of $\tau$ according to Figure \ref{koutenpn}.
Then, it is known \cite{CJKLS} that, thanks to \eqref{kis22o},
the map $\mathcal{I}_{\Phi}$ is independent of the choice of $D$ by \eqref{kis22o};
then, $\mathcal{I}_{\Phi}$ is called {\it the quandle cocycle invariant}.
For example, given an additive homomorphism $\psi:X^2 \ra A$ satisfying $\psi (tx, ty)=\psi (x,y)$ for any $x,y \in X$, %\eqref{ },
we can easily verify that the map $\phi_{\psi}$ defined by $\phi_{\psi}(x,y)= \psi (x-y,y-yt^{-1}) $ satisfies \eqref{kis22o},
and, by definitions, that the associated invariant $\mathcal{I}_{\Phi} $
is equal to
the restricted 2-form $\mathcal{Q}_{\psi} \circ \bigtriangleup$, where $\bigtriangleup $ is the diagonal map $ \mathrm{Col}_X (D)\ra \mathrm{Col}_X (D)^2$ and $\mathcal{Q}_{\psi}  $ with $X=M=M'$ is defined in \eqref{kiso}.
% the original quandle cocycle invariant.

%As mentioned in Remark \ref{als1p32w}, for every $\Z[t^{\pm 1}]$-invariant bilinear form $\psi: X^2 \ra A$,further,

% ensures

%Before going to the next section, we give a corollary.
%Incidentally, as a result, we will describe an equivalence between the cocycle invariant and the $S$-equivalence; see Appendix \ref{}.
%Recall from \cite[Remark \ref{cl31l}]{Nos5} that the diagonal restriction $\mathcal{Q}_{\psi} \circ \bigtriangleup $
%is the quandle cocycle invariant $\mathcal{I}_{\Psi}$ which is originally defined in \cite{CJKLS}.
%??????????

Similar to Theorem \ref{clAl22}, we will show the $S$-equivalence and a universality of the quandle 2-cocycle invariants.

\begin{thm}\label{clA21l31}
If two knots $K$ and $ K'$ are $S$-equivalent, then
for every Alexander quandle $X$ and every quandle 2-cocycle $\phi: X^2 \ra A$,
the associated cocycle invariants $\mathcal{I}_{\Phi} $ and $\mathcal{I}_{\Phi}' $ are equivalent.

Furthermore, for a knot $K$, there is an Alexander quandle $X_0$ and a bilinear map $\psi_{X_0} : X_0 \times X_0 \ra X_0$ such that, for any Alexander quandle $X$, any quandle 2-cocycle $\phi: X^2 \ra A$ and any $X $-coloring $\mathcal{C}$ of $K$,
there are an $X_0$-coloring $ \mathcal{C}_0$ and an additive homomorphism
$\mathcal{P}_{\phi}: X_0 \ra A$ such that $ \mathcal{P}_{\phi} \bigl( \mathcal{Q}_{ \psi_{X_0}} ( \mathcal{C}_0, \mathcal{C}_0) \bigr) = \mathcal{I}_{ \Phi} (\mathcal{C} ) $.
%Let $D$ be a diagram of a knot $K$ with $\# K =1$.
%(I) Then we have an $\Z[T^{\pm 1 } ] $-module isomorphism
%$$ \mathrm{Col}_X(D) \cong H^1( \pi_L, \partial \pi_L ;X ) \oplus X. $$
%Moreover, if $X$ is connected, this $ H^1( \pi_L, \partial \pi_L ;X) $ is the usual cohomology $ H^1( \pi_L; X ) $.
%\noindent (II) If $L$ is a knot $K$,
%Denote by $A_K$ by the Alexander module of $K$, and $\Z[T^{\pm 1}]$ by the Laurent polynomial ring $\Z[T^{\pm 1}]$.
\end{thm}
% way original
% We will aim to show this theorem.
%e proof.
%The remaning
%To prove this theorem, we now show the following proposition:

In conclusion, this theorem implies that every 2-cocycle invariants from Alexander quandles
can be described from the Blanchfield pairing $ \mathrm{Bl}_K $ in principle.
While the cocycle invariants are diagrammatically defined,
this theorem implies a topological interpretation of the 2-cocycle invariants in the sense of cohomology pairings $\mathcal{Q}_{\psi}$.
%\begin{prop}\label{clAl} %Let $D$ be a diagram of a knot $K$.
% and Alexander polynomial $\Delta_K$.
%\noindent
%(I) Then we have an $\Z[T^{\pm 1 } ] $-module isomorphism
%$$ \mathrm{Col}_X(D) \cong H^1( \pi_L, \partial \pi_L ;X ) \oplus X. $$
%Moreover, if $X$ is connected, this $ H^1( \pi_L, \partial \pi_L ;X) $ is the usual cohomology $ H^1( \pi_L; X ) $.
%\noindent (II) If $L$ is a knot $K$,
%Denote by $A_K$ by the Alexander module of $K$, and $\Z[T^{\pm 1}]$ by the Laurent polynomial ring $\Z[T^{\pm 1}]$.
%Theorem \ref{clAl} on
%the equivalence between the bilinear form $\mathcal{Q}_{\psi}$ and the Blanchfield pairing.
\begin{proof}[Proof of Theorem \ref{clA21l31}]
%[Proof of Proposition \ref{clAl} (I)-(III)]
%[Proof of Theorem \ref{clAl}]
%(I) is immediately shown by Theorem \ref{}. %obtained from
%Let us start by showing (1).
As the first step, we claim that, for any bilinear function $\psi: X^2 \ra A$ satisfying $\psi(x,y)=\psi(ty, x)$, the associated bilinear form $\mathcal{Q}_{\phi_\psi} $ is an $S$-equivalent knot invariant.
Notice that $\psi$ satisfies $ \psi (x,y)=\psi(ty, x)= \psi(tx,ty)$. Therefore,
in the same way as the proof of Theorem \ref{cl022}, the claim can be easily shown.
%It is enough for (i) to show the local invariance with respect to the double delta move. Given an $X $-coloring of $D$, consider $X $-colors of the arcs $\alpha_i$ with $1\leq i \leq 12$ in Figure \ref{fig.color22}. By direct calculation, it can be seen that the correspondence $\alpha_i \leftrightarrow \beta_i$ ensures another coloring of $D'$, which yields a $ \Z[t^{\pm 1}]$-module isomorphism $\mathcal{B}: \mathrm{Col}_X(D)\cong \mathrm{Col}_X(D'_f ) $.

We will deal with any quandle 2-cocycle $\phi: X^2 \ra A$.
% that stated an isomorphism ${\rm Col}_X(D) \cong \Hom ( H_1( \pi_K ;\Lambda ),X) $.
%Recall the bijection ${\rm Col}_X(D)\simeq \mathrm{Hom}_{\mathrm{Qnd}} ( Q_K, X)$; see \cite{Joy}.
According to \cite[Theorem 17.3]{Joy} concerning %, the quotient of $Q_K$ by
``Abelianization of the knot quandle", %, on which we give a theory in Appendix,
for any Alexander quandle $X$, there is a functorial $\Z[t^{\pm 1}]$-module isomorphism
$$ \mathrm{Col}_X (D) \cong \mathrm{Hom}_{\Z[t^{\pm 1}] \textrm{-mod}} \bigl( H_1( S^3 \setminus K;\Z[t^{\pm 1}]), \ X \bigr) \oplus X . $$
In other words, $\mathrm{Col}_X (D)$ is representable by the Alexander module $H_1( S^3 \setminus K;\Z[t^{\pm 1}])$.
%is isomorphic to the Alexander quandle on the $\Lambda$-module $H_1( \pi_K ;\Lambda ) $, in which $1-t $ is invertible.
%\footnote{This invertibitily is obtained from the known fact that the special value of the Alexander polynomial at $T=1$ is 1.}.
%Hence the set $\mathrm{Hom}_{\mathrm{Qnd}} ( Q_K, X)$ is bijective to $\mathrm{Hom}_{\mathrm{Qnd}} ( H_1( \pi_K ;\Lambda ), X) $.
%Consequently, the desired $\Lambda $-module isomorphism is immediately obtained from Lemma \ref{lem3ale} below.
Hence, by the universality, we may assume $X= H_1( S^3 \setminus K;\Z[t^{\pm 1}])$ hereafter.
In particular, $(1-t)X=X$.

Further,
consider ``the quandle second homology $H_2^Q(X;\Z)$" defined in \cite{CJKLS},
which is isomorphic to the quotient $\Z$-module
$$ G_X:= X \otimes_{\Z} X / \{ x \otimes y - (t y) \otimes x \ | \ x, y \in X\ \}.$$
This result is essentially due to Clauwens \cite{Cla} (see also \cite[\S 5]{Nos1} or \cite{BKMNP} for the proof). %, the homology is
Further, consider a homomorphism $ \psi_{\rm uni}: X \otimes X \ra G_X$ which sends $ (x,y)$ to $ [x \otimes y]$.
%Next, we will explain a universal
Then, it is known (see, e.g., \cite{CJKLS,Nos1}),
every cocycle invariant $ \bigtriangleup \circ \mathcal{Q}_{\psi}$ factors through
the homology $H_2^Q(X;\Z)$. Precisely,
there is an additive homomorphism $P_{\psi}: G_X \ra A$ such that $ \mathcal{Q}_{ \psi_{\rm uni}}\circ \bigtriangleup = P_{\psi} (\mathcal{Q}_{ \psi} \circ \bigtriangleup) $ for any link $L$.
By the discussion in the first paragraph, $ \mathcal{Q}_{ \psi_{\rm uni}} $ is $S$-equivalent; hence, so is $\mathcal{Q}_{ \psi}$ for any quandle 2-cocycle $\psi.$
%Hence, from the definitions of $G_X$ and $ \psi_{\rm uni}$,the proof for the $S$-equivalence is done in the same manner as the first step.

To show the latter part, let $X_0$ be $ H_1( S^3 \setminus K;\Z[t^{\pm 1}]) $ as above.
Since $X_0$ is finitely generated over $\Lambda$, we can make $X_0$ into a $\Lambda$-algebra. Then, we let $\psi_{X_0} :X_0\times X_0 \ra X_0$ send $(x,y) \mapsto \bar{x}y $, which factors through $G_{X_0}$ via $\psi_{\rm uni}$.
Therefore, by the discussion in the previous paragraph, any cocycle invariant $\mathcal{I}_{\Phi} $ is derived from $\mathcal{I}_{\Phi_{X_0}} = (\mathcal{Q}_{ \psi_{\rm uni}} \circ \bigtriangleup) $.
This means the desired statement.
% Let $ $
%(IV).
%As the setting,
%we let $ X= \Lambda /(\Delta_K )$ and $ \psi : X^2 \ra A$ be by setting $ \psi (x,y): = \bar{x}y -\bar{y}x$.
%For this, let $A$ and $X$ be the quotient ring $ \Lambda /(\Delta_K ) $ subject to the first Alexander polynomial
%$ \Delta_K $.
%Noting the famous fact $\Delta_K(t^{-1})=\Delta_K(t) $,
%we define a $ \Z$-invariant skew-hermitian bilinear form
%$ \psi : X^2 \ra A$ by setting $ \psi (x,y): = \bar{x}y -\bar{y}x$.
% setting $ \psi_{\rm pre }:\Z^2 \ra \Z $ as the multiplication,
%we have a $ \Lambda$-invariant bilinear form
%$\psi :X^2 \ra A$ such as \eqref{}
%Then
%Theorem \ref{mainthm2} and Proposition \ref{pro8} below immediately lead to the desired equivalence.
\end{proof}

\subsection*{Acknowledgment}
The author sincerely expresses his gratitude to Anthony Conway for giving him valuable comments and many questions.
% on an early version of the paper.
The work was partially supported by JSPS KAKENHI Grant Number 00646903.

\vskip 1pc

\normalsize

\noindent
Department of Mathematics, Tokyo Institute of Technology
2-12-1
Ookayama, Meguro-ku Tokyo 152-8551 Japan

\end{document}